\documentclass{amsart}
\usepackage{amssymb}
\usepackage{amsfonts}
\usepackage{amsmath}
\usepackage[legalpaper,bookmarks=true,colorlinks=true,linkcolor=blue,citecolor=blue]{hyperref}
\usepackage{graphicx}%
\usepackage{fancyhdr}
\usepackage{color}
\usepackage[mathlines]{lineno}
\usepackage{lscape}
\usepackage{epsfig}
\usepackage{natbib}
\usepackage{geometry}
\usepackage{tgbonum}
\usepackage[]{listings}
\usepackage{soul} 

\fontfamily{qcr}\selectfont

\newtheorem{theorem}{Theorem}
\theoremstyle{plain}

\numberwithin{equation}{section}

\newcommand{\Bin}{\bigskip \noindent}

\newcommand{\Ni}{\noindent}

\begin{document}
\Large
\title[Pseudo-Lindley Alpha Power transformed distribution]{The Pseudo-Lindley Alpha Power transformed distribution, mathematical characterizations and asymptotic properties }

\author{Modou Ngom $^{\dag}$}
\author{Moumouni Diallo $^{\dag\dag}$}
\author{ Adja Mbarka Fall $^{\dag\dag\dag}$}
\author{Gane Samb Lo $^{\dag\dag\dag \dag}$}

\begin{abstract} We introduce a new generalization of the Pseudo-Lindley distribution by applying alpha power transformation. The obtained distribution is referred as the Pseudo-Lindley alpha power transformed distribution (\textit{PL-APT}). Some tractable mathematical properties of the \textit{PL-APT} distribution as reliability, hazard rate, order statistics and entropies are provided. The maximum likelihood method is used to obtain the parameters' estimation of the \textit{PL-APT}  distribution. The asymptotic properties of the proposed distribution are discussed. Also, a simulation study is performed to compare the modeling capability  and flexibility of \textit{PL-APT} with Lindley and Pseudo-Lindley distributions. The \textit{PL-APT} provides a good fit as the Lindley and the Pseudo-Lindley distribution. The extremal domain of attraction of \textit{PL-APT} is found and its quantile and extremal quantile functions studied. Finally, the extremal value index is estimated by the double-indexed Hill's estimator (Ngom and Lo, 2016) and related asymptotic statistical tests are provided and characterized.\\
 
\noindent Modou Ngom $^{\dag}$\\
Work Affiliation : Ministry of High School (SENEGAL)\\
LERSTAD, Gaston Berger University, Saint-Louis, S\'en\'egal\\
Imhotep Mathematical Center\\
Email:ngom.modou1@ugb.edu.sn, ngomodoungom@gmail.com.\\

\noindent Moumouni Diallo $^{\dag\dag}$\\
Work Affiliation : Universit\'e des Sciences Sociale et de Gestion de Bamako(USSGB)\\
Imhotep Mathematical Center\\
Email: moudiallo1@gmail.com\\

\noindent Adja Mbarka Fall \noindent $^{\dag\dag\dag}$\\
Work Affiliation :  Universit\'e Iba Der Thiam de Thi\'es (UIDT)\\
Imhotep Mathematical Center\\
LERSTAD, Gaston Berger University, Saint-Louis, S\'en\'egal\\
Email:adjambarka.fall@univ-thies.sn\\

\noindent Gane Samb Lo \noindent $^{\dag\dag\dag \dag}$\\
LERSTAD, Gaston Berger University, Saint-Louis, S\'en\'egal (main affiliation)\\
LSTA, Pierre and Marie Curie University, Paris VI, France\\
AUST - African University of Sciences and Technology, Abuja, Nigeria\\
Imhotep Mathematical Center\\
Email:gane-samb.lo@edu.ugb.sn, gslo@aust.edu.ng, ganesamblo@ganesamblo.net\\
Permanent address : 1178 Evanston Dr NW T3P 0J9,Calgary, Alberta, Canada.\\

\noindent\textbf{ MSC2020-Mathematics Subject Classification System :} 60G70; 62G20; 62E15; 62F12; 60F05.

\end{abstract}
\maketitle

\section{Introduction} \label{intro}

\Bin In the last decades, the Lindley distribution (see \cite{lindley1958, lindley1965}), with parameter $\theta > 0 $, has been a center of interests of many research activities. The family of Lindley distribution of one parameter $\theta >0$ has the following cumulative distribution function (\textit{cdf}) 

\begin{equation} \label{Cdf_lindl}
F_{L}(x)=\left(1- \left( 1 + \frac{\theta x}{1+\theta} \right) \exp(-\theta x) \right)1_{(x\geq 0)}.
\end{equation} 

\Bin The corresponding probability distribution function (\textit{pdf}) of (\ref {Cdf_lindl}) is given by

\begin{equation} \label{Pdf_lindl}
f_{L}(x)=\left(\frac{\theta ^{2}}{1+\theta} \left( 1 + x \right) \exp(-\theta x) \right)1_{(x\geq 0)}. 
\end{equation} 

\Bin Lindley's statistical distribution, which has been proposed as an alternative model to fit data with non-monotone hazard rate, and its different generalizations  have attracted a great attention from researchers. This is justified by the importance of such distribution in many areas as reliability for example. Let us cite a few number of important examples. A new bounded domain probability density  feature in view of a generalized Lindley distribution is considered in  \cite{Ghitanyetal}. The Lindley distribution has been used for modeling completing risks in lifetime data in \cite{MazucheliAchcar2011}. A statistical inference on the parameter in its progressive Type-II censoring scheme is provided in \cite {KrishnaKumar2011}. \cite{Gomezetal} use the Log-Lindley distribution in the application of strength systems reliability in the field of insurance and inventory management. A comparison study of the adequacy of exponential and Lindley distributions on modeling of lifetime data is studied by \cite{Shankeretal2015}. A study is carried out by \cite{Hafezetal}, using the accelerated life tests under censored sample and the importance of the distribution is introduced applying an experimental application. 

\Bin  However, the Lindley distribution has an increasing failure rate and this makes it not flexible in lifetime data  modeling. Because of the above mentioned importance, a significant number of generalizations has been introduced to improve its ability to analyze various types of lifetime data with a high degree of skewness and kurtosis. The later generalizations continue themselves to be extended.  The idea is to correct this flaw by increasing the number of parameters. Indeed, most of the generalizations introduced other parameters in hope of capturing the complexity data in  lifetime data. One of these direct generalizations is developed in \cite{zeghdoudi2016} and is named as the Pseudo-Lindley distribution with two parameters $\theta>0$ and $\beta>1$ . The \textit{cdf} and \textit{pdf} of  the Pseudo-Lindley distribution  are defined respectively by 

\begin{equation}
F_{PL}(x)= \left( 1-\beta ^{-1}\left( \beta +\theta x\right) \exp \left( -\theta x\right) \right)1_{(x\geq 0)}, \label{CdfPseudolind}
\end{equation}

\begin{equation}
f_{PL}(x)= \left(\frac{\theta \left( \beta -1+\theta x\right) \exp \left( -\theta
x\right) }{\beta }\right) 1_{(x\geq 0)}.  \label{PdfPseudolind}
\end{equation}

\Bin Further statistical studies on the Pseudo-Lindley distribution are available in \cite{gslo2020, gslo2019} which focused on the asymptotic theory of moments estimators, the extreme values characterization and estimations with among other topics. Also, a discrete version of the Pseudo-Lindley distribution is proposed by \cite{Irshadetal} with a stress on  the mathematical properties.

\Bin On another side, \cite{mahdavi2017} introduced a powerful method of creating new statistical distributions named  the alpha power transformed (\textit{APT}). For any \textit{cdf} $F$ with respective lower and upper endpoints

$$
lep(F)=\inf \{x \in \mathbb{R}, \ F(x)>0\} \ \ , \ \ uep(F)=\sup \{x \in \mathbb{R}, \ F(x)>0\},
$$

\Bin its \textit{APT} $G_{\alpha}$  is defined, for $\alpha \in ] 0, +\infty [\setminus \{1\}$ as 

\begin{equation}\label{aptgcdf}
G_{\alpha}(x)=\frac{1-\alpha ^{F\left( x\right) }}{1-\alpha}, \ x \in [lep(F), \ uep(F)].
\end{equation}

\Bin  We may seet that $G_{\alpha}$ is a \textit{cdf} by considering the two cases $0<\alpha<1$ and $\alpha>1$. In fact, by using the non-decreasingness of $F$, we have, for $\alpha>1$, that $\alpha^{F(x)}=\exp(F(x) \log \alpha)$ is non-decreasing and hence $1-\alpha^{F(x)}$ non-increasing. Since the denominator $1-\alpha$ is negative, we get that $G_{\alpha}$ is non-decreasing. A similar method shows that $G_{\alpha}$ is still non-decreasing for $0<\alpha<1$. Besides, $G_{\alpha}$ is right-continuous and $\lim_{x\rightarrow lep(F)} G_{\alpha} (x) = 0 $ and $\lim_{x\rightarrow uep(F)} G_{\alpha}(x) = 1 $.

\Bin Furthermore, we have $lep(F) = lep(G_{\alpha})$ and $ uep(F) = uep(G_{\alpha})$. The definition may be extended to  $\alpha=1 $ by taking $ G_{1}=F$.

\Bin Whenever $F$ has a probability density function ($ pdf $) $f$,the \textit{APT} $G_\alpha$, has the $ pdf $ defined as follows

\begin{equation}\label{aptgcdf}
g_{\alpha}(x)=\frac{\log \left( \alpha \right) }{\alpha -1}f\left( x\right) \alpha^{F(x)} , \ \ x \in [lep(G_{\alpha}), \ uep(G_{\alpha})],
\end{equation}

\Ni for $\alpha \in ]0, +\infty[\setminus \{1\}$ and $g_{1}=f$.

\Bin One of the reasons of appealing to the \textit{APT} is its ability to make distribution more flexible to fit correctly and adequately some lifetime data. For example the following uses of the \textit{APT} have been made:  \cite{Aldahlan} for log-logistic distributions, \cite{ZeinEldinetal} for the Inverse Lomax distribution, \cite{eghwerido} for the Teissier distribution and \cite{Ijazetal} for exponential distribution, to cite a few.\\

\Ni The alpha power transformed quasi Lindley distribution (\textit{APTQL}) has been studied in \cite{unyimeette} using the quasi Lindley distribution introduced by \cite{ShankerMishra2013}. The \textit{cdf} of the \textit{APTQL} is defined for $\beta >-1$ and $\theta >0$ by 

\begin{equation}\label{aptqldcdf0}
F_{\alpha}(x)= \frac{1}{1-\alpha}\left( 1-\alpha ^{1-\left(\frac{\beta +1+\theta x}{\beta +1}\right) \exp \left(
-\theta x\right) }\right), \ \ x \geq 0 ,
\end{equation}

\Bin if $\alpha \in ] 0, +\infty [\setminus \{1\}$ and 

\begin{equation}\label{aptqldcdf1}
F_{1}(x)= 1-\left(\frac{\beta +1+\theta x}{\beta +1}\right) \exp \left(
-\theta x\right),\ \ x \geq 0. 
\end{equation}

\Bin The aim of this paper is contributing  to the current trend by precisely applying the \textit{APT} to the Pseudo-Lindley distribution (see \cite{zeghdoudi2016} ), already cited above. We will show that the new class of distributions, called the Pseudo-Lindley Alpha Power Transformed distribution (\textit{PL-APT}), can be used to  improve the flexibility of continuous real lifetime data over the Pseudo-Lindley distribution. We study the statistical properties which will be summarized in a reserved paragraph \pageref {mathproperty}.

\Bin The rest of the paper is organized as follows. In Section \ref {model}, we present the model of the \textit{PL-APT} such as the \textit{cfd} and the \textit{pdf}. We precise the type of distribution when the parameters take some particular values.Section \ref {mathproperty} is devoted to some statistical properties of the \textit{PL-APT} related reliability and hazard rate functions, order statistic and entropies and likelihood method of parameters estimations. The asymptotic properties of the new family are presented in Section \ref{asymproperty}. Its quantile function and the related extreme expansions are studied in Subsection \ref{quantextremquant}  and the extremal behaviors in Subsection \ref{extremes}, where the doublei-indexed \cite{ngomlo2016}'s statistic is used to estimate the extreme index value and its asymptotic law is expanded. The proof of the expressions of its quantile function and its expansions are stated in an appendix from page \pageref{pageAppendixA1}. Section \ref{conclud} concludes the paper and gives perspectives.   

\section{The model} \label{model}
\Bin Let $ X $ be a random variable following a \textit{PL-APT} with parameters $\alpha$, $\beta$ and $\theta$ denoted by $X\hookrightarrow$ \textit{PL-APT} $(\alpha,\beta,\theta)$. According to equations (\ref{CdfPseudolind}) and (\ref{aptgcdf}) the \textit{cdf} of the \textit{PL-APT} is defined as follows :
\Bin if $\alpha \in ] 0, +\infty [\setminus \{1\}$,

\begin{equation} 
G_{\alpha}(x)=\frac{1}{1-\alpha} \left(1-\alpha ^{1-\beta ^{-1}\left( \beta +\theta x\right) \exp \left(
-\theta x\right) } \right) 1_{(x\geq 0)}, \label{CdfAlpaPseudolin2}
\end{equation}

\Bin and  

\begin{equation} 
G_{1}(x)= \left(1-\beta ^{-1}\left( \beta +\theta x\right) \exp \left( -\theta x\right) \right) 1_{(x\geq 0)}.\label{CdfAlpaPseudolin1}
\end{equation}

\Bin From equations (\ref{CdfAlpaPseudolin2}), (\ref{CdfAlpaPseudolin1}) and (\ref{aptgcdf}), we find the \textit{pdf} of the \textit{PL-APT} defined by 

\begin{equation} 
g_{\alpha}(x)= \left(\frac{\theta \log \left( \alpha \right) \left( \beta -1+\theta x\right) \exp
\left( -\theta x\right) }{\beta \left( \alpha -1\right) }\alpha ^{1-\beta
^{-1}\left( \beta +\theta x\right) \exp \left( -\theta x\right) } \right) 1_{(x\geq 0)}, \label{PdfAlpaPseudolind2}
\end{equation}

\Bin if  $\alpha \in ] 0, +\infty [\setminus \{1\}$, and 

\begin{equation} 
g_{1}(x)= \left(\frac{\theta \left( \beta -1+\theta x\right) \exp \left( -\theta
x\right) }{\beta }\right) 1_{(x\geq 0)}. \label{PdfAlpaPseudolind1}
\end{equation}

\begin{enumerate}
	\item  If the parameter $\alpha = 1 $ then the \textit{PL-APT} distribution corresponds to the Pseudo-Lindley distribution
developed by \cite{gslo2020}. The figure \ref{applalpheg1}
shows the graphs of the \textit{cdf} and the \textit{pdf} with several values of the
parameters $ \alpha $, $ \beta$ and $ \theta $.

\begin{figure}[htbp]
\begin{minipage}[c]{.46\linewidth}
\includegraphics[height=8cm, width=8cm]{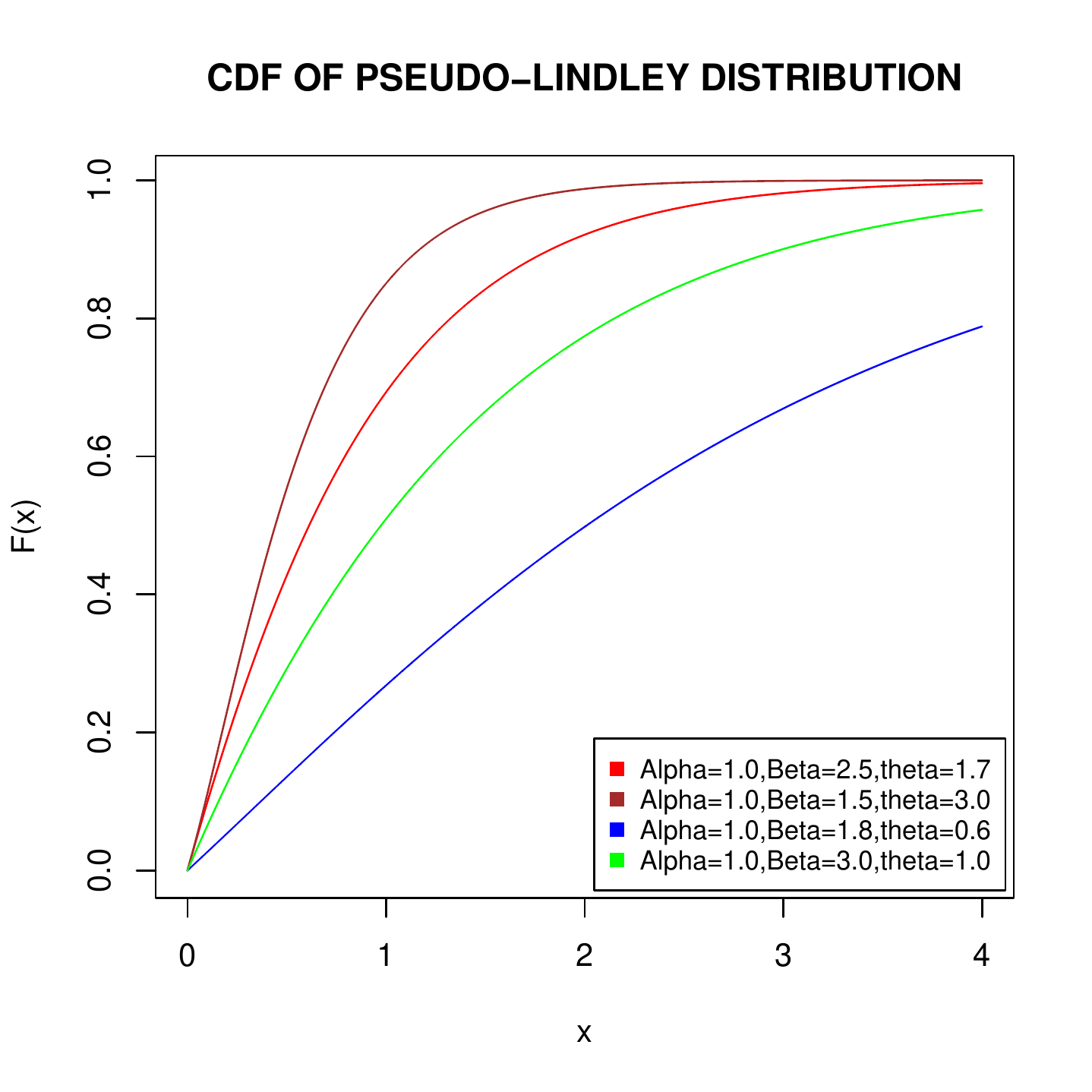}
\end{minipage} \hfill 
\begin{minipage}[c]{.46\linewidth}
\includegraphics[height=8cm, width=8cm]{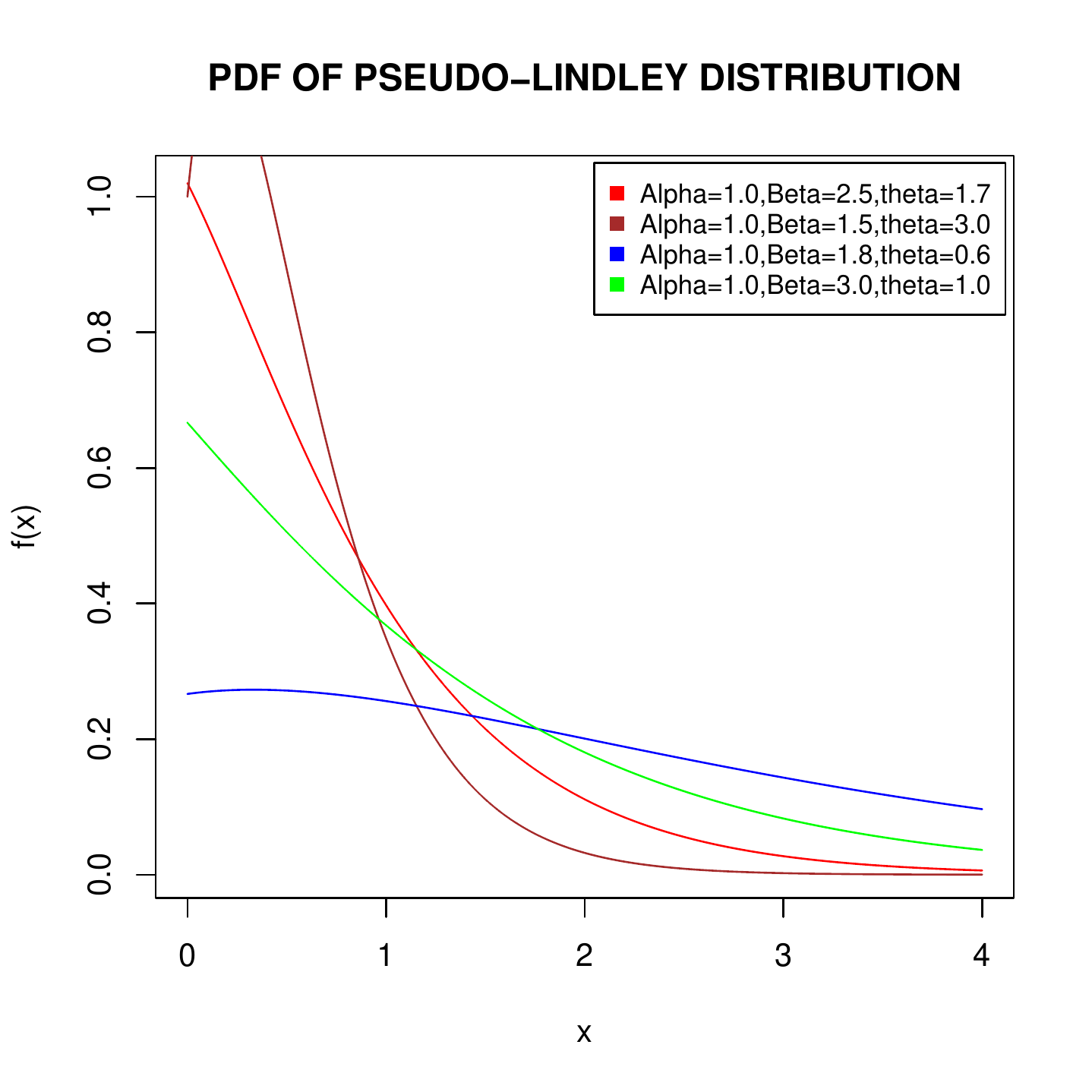}
\end{minipage}
\caption{Graphs of the \textit{cdf}  (left) and \textit{pdf}  (right) for the Pseudo-Lindley distribution with several values of parameters $\alpha$, $\beta$ and $\theta$.}
\label{applalpheg1}
\end{figure}


\item  If the parameter $\alpha = 1 $, $\beta=1+\theta$ and $ \theta > 0 $ then the \textit{PL-APT} distribution corresponds to the Lindley distribution developed by Lindley in the two papers  \cite{lindley1958, lindley1965}. The figure \ref{applalpheg1lind}
shows the graphs of the \textit{cdf}  and the \textit{pdf} with several values of the
parameters $ \alpha $, $ \beta$ and $ \theta $.

\begin{figure}[htbp]
\begin{minipage}[c]{.46\linewidth}
\includegraphics[height=8cm, width=8cm]{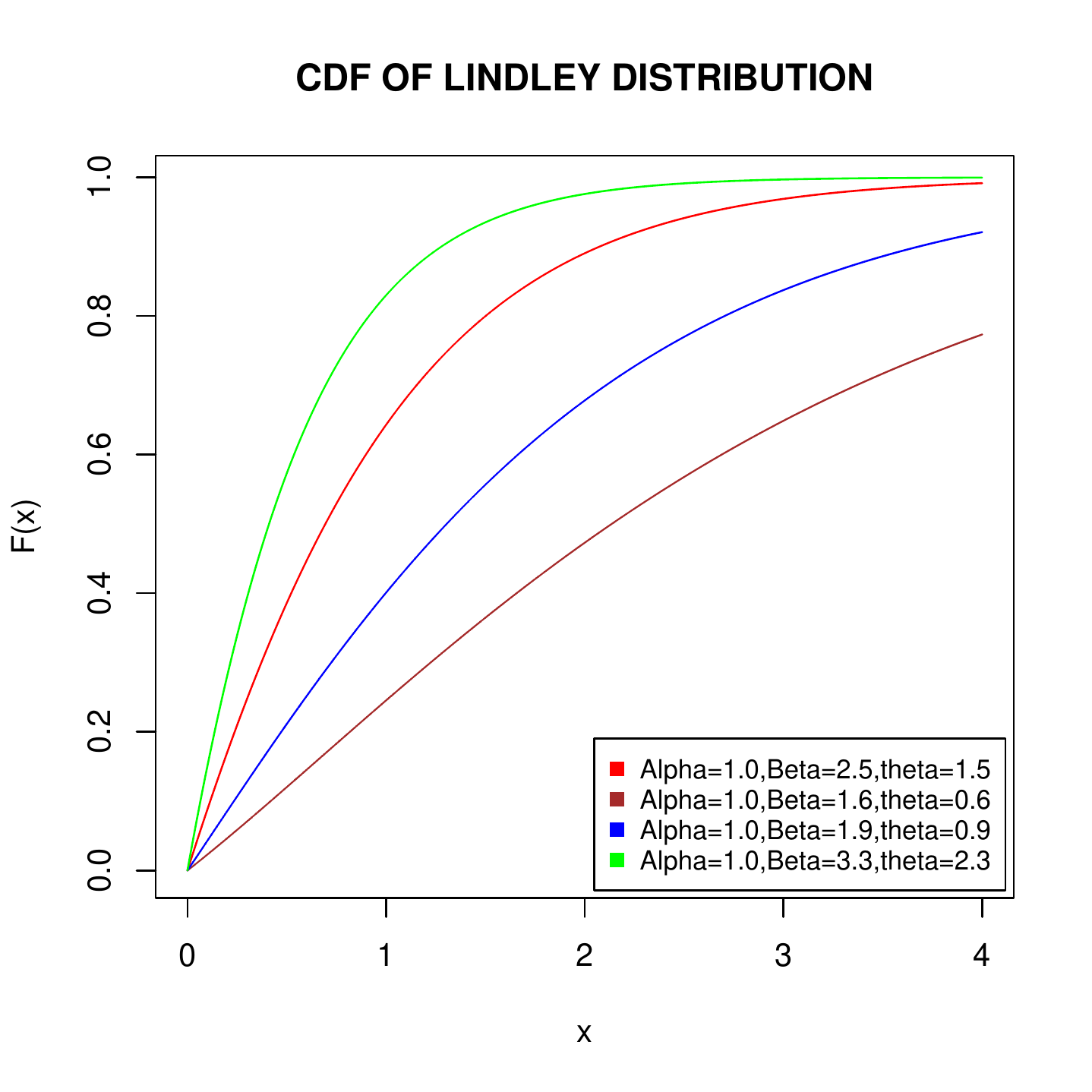}
\end{minipage} \hfill 
\begin{minipage}[c]{.46\linewidth}
\includegraphics[height=8cm, width=8cm]{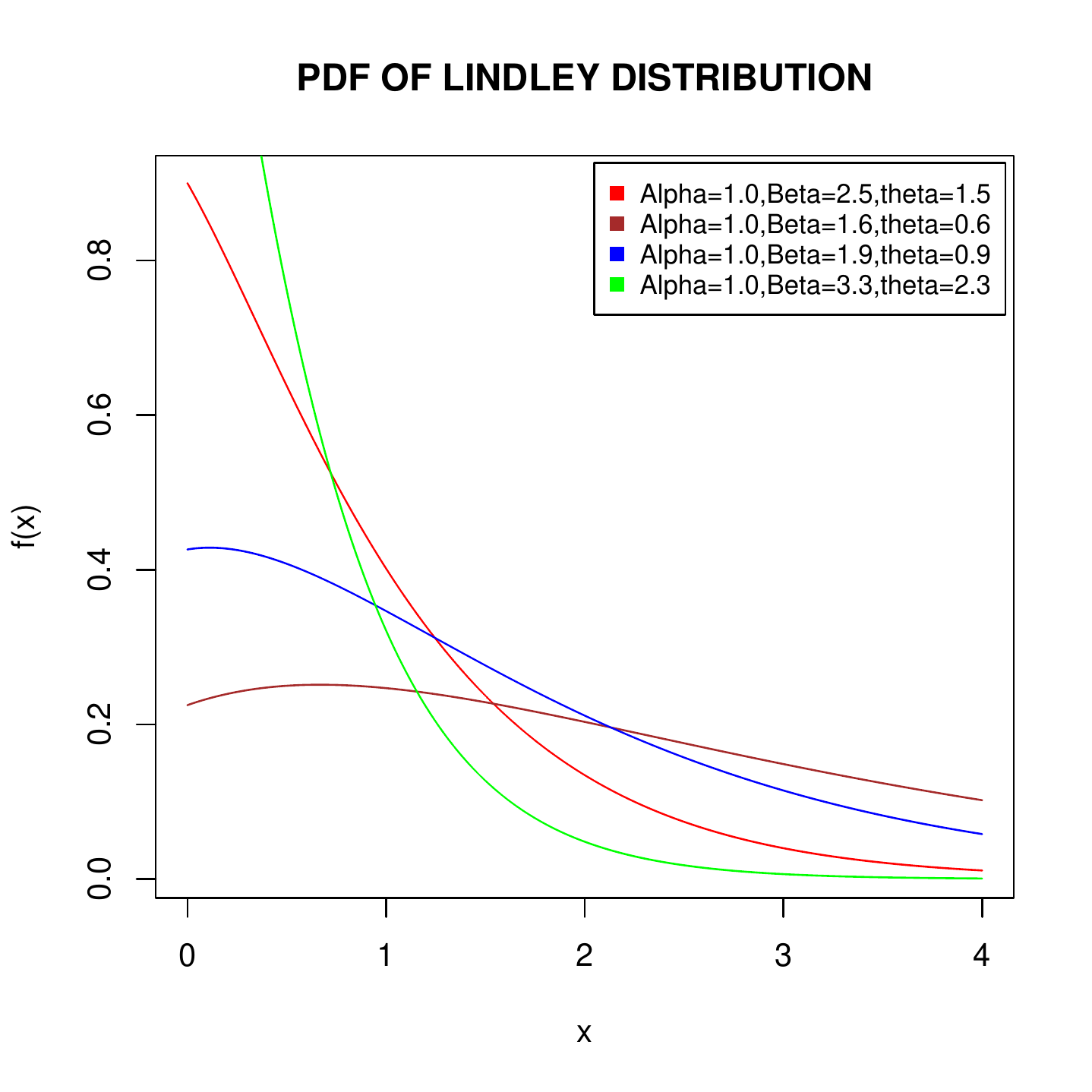}
\end{minipage}
\caption{Graphs of the \textit{cdf}  (left) and \textit{pdf}  (right) for the Lindley distribution with several values of parameters $\alpha$, $\beta$ and $\theta$.}
\label{applalpheg1lind}
\end{figure}


\item  If the parameter $\alpha \in ] 0, +\infty [\setminus \{1\}$ then the \textit{cdf}  and the 
\textit{pdf}  of the \textit{PL-APT} distributions are represented by the graphs  of figure \ref%
{applalphdif1} below for several values of the parameters $ \alpha $, $ \beta$ and $ \theta $.

\begin{figure}[htbp]
\begin{minipage}[c]{.46\linewidth}
\includegraphics[height=8cm, width=8cm]{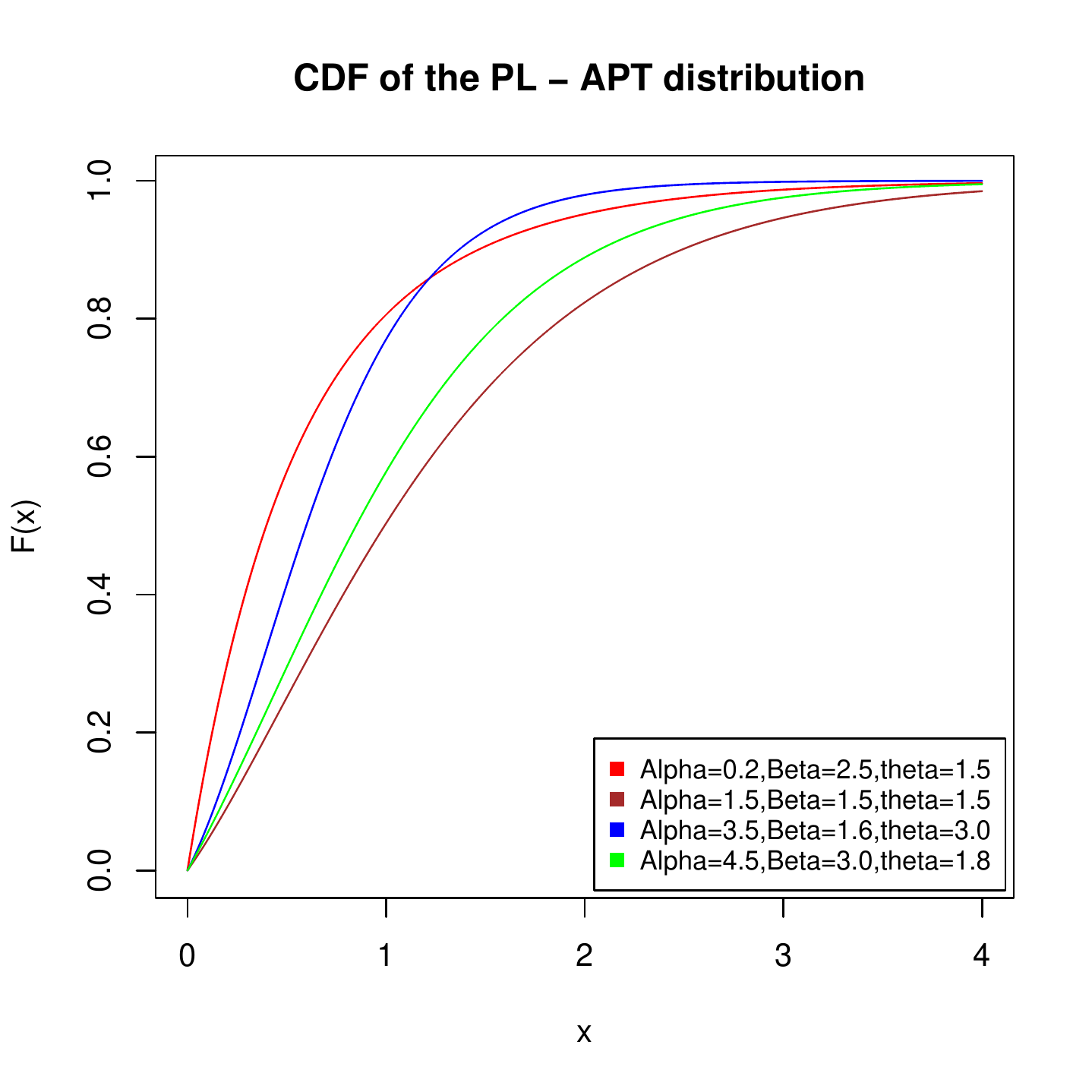}
\end{minipage} \hfill 
\begin{minipage}[c]{.46\linewidth}
\includegraphics[height=8cm, width=8cm]{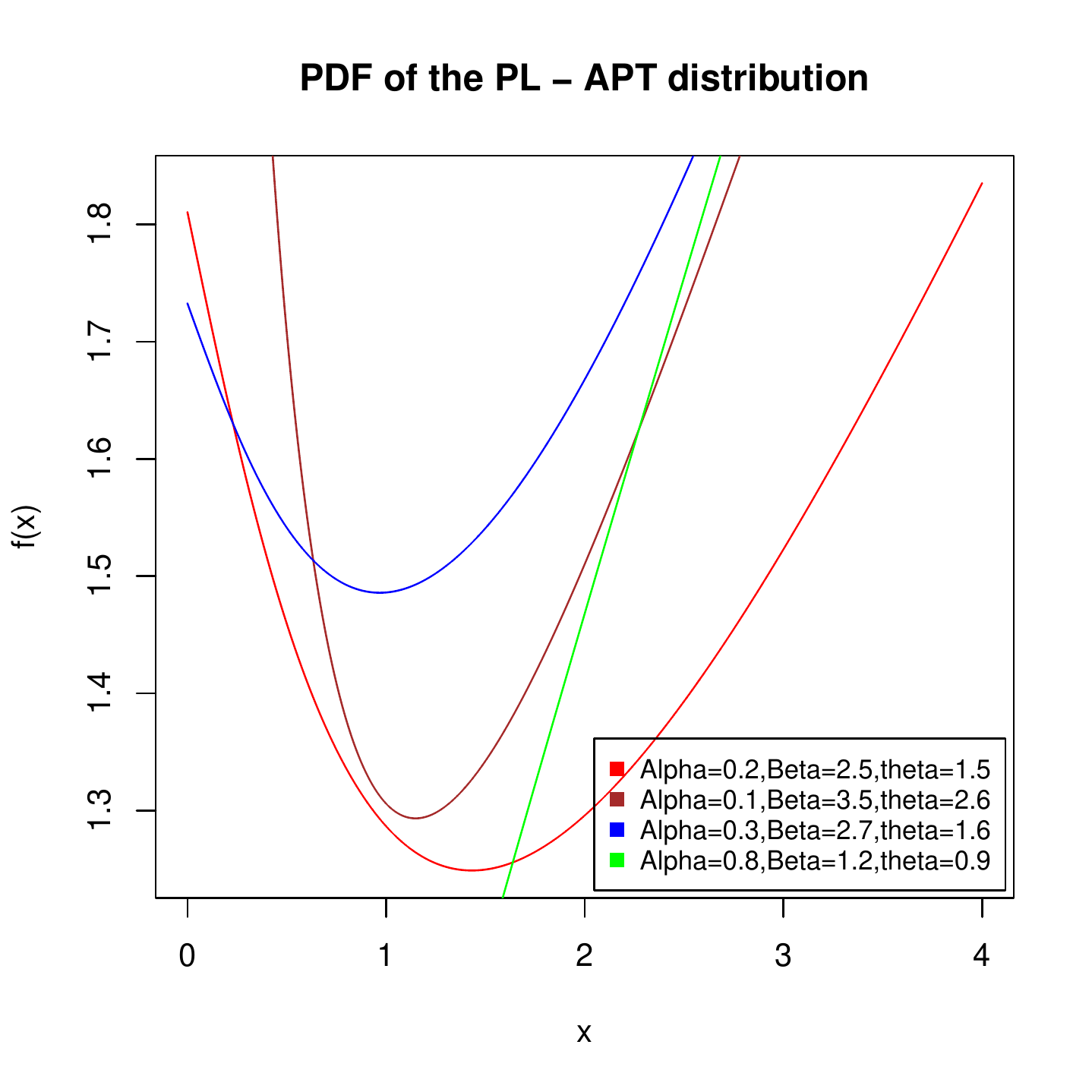}
\end{minipage}
\caption{Graphs of the \textit{cdf} (left) and \textit{pdf} (right) for the \textit{PL-APT} distribution with several values of parameters $\alpha$, $\beta$ and $\theta$.}
\label{applalphdif1}
\end{figure}

\end{enumerate}

\section{Mathematical properties} \label{mathproperty}
\Bin Some basic statistical properties of the \textit{PL-APT} distribution with parameters $\alpha$, $\theta >0$
and $\beta >1$ are derived and established in this section.
\subsection{Reliability}
\Bin The reliability function of the \textit{PL-APT} distribution is expressed as follows
\Bin for $\alpha \in ] 0, +\infty [\setminus \{1\}$, 

\begin{eqnarray*} \label{Reliability2}
R_{\alpha}(x) &=& 1-G_{\alpha}\left( x\right) \\
     &=& 1-\frac{1-\alpha ^{1-\beta ^{-1}\left( \beta +\theta x\right) \exp \left(
-\theta x\right) }}{1-\alpha }\\
    &=& \frac{\alpha}{\alpha-1} \left( 1-\alpha^{-\beta^{-1}(\beta+\theta x) \exp(-\theta x)}\right),
\end{eqnarray*}%

\Bin $ x \in [lep(F), \ uep(F)]$ and for $x \in [lep(F_{1}), \ uep(F_{1})] $, 

\begin{eqnarray*} \label{Reliability1}
R_{1}(x) &=&1-G_{1}\left( x\right)\\
    &=& \beta ^{-1}\left( \beta +\theta x\right) \exp \left( -\theta x\right).
\end{eqnarray*}%

\subsection{Hazard rate function}

\Bin The mathematical formula for hazard function which is
otherwise called failure rate is defined as follows:

\Bin if $\alpha \in ] 0, +\infty [\setminus \{1\}$, 

\begin{equation*}
h_{\alpha}(x)= \frac{g_{\alpha}\left( x\right) }{1-G_{\alpha}\left( x\right)}, \ x \in [lep(F), \ uep(F)]
\end{equation*}

\Bin and  

\begin{equation*}
h_{1}(x)= \frac{f\left( x\right) }{1-F\left( x\right)}, \ x \in [lep(F), \ uep(F)].
\end{equation*}

\Bin So, we obtain the hazard function of the \textit{PL-APT} distribution with parameters $\alpha$, $\theta >0$ and $%
\beta >1$ as follows:\\

\Bin if $\alpha \in ] 0, +\infty [\setminus \{1\}$, 

\begin{equation} 
h_{\alpha}(x)= \frac{\beta ^{-1}\theta \log \left( \alpha \right) \left( \beta -1+\theta
x\right) \exp \left( -\theta x\right) \alpha ^{1-\beta ^{-1}\left( \beta
+\theta x\right) \exp \left( -\theta x\right) }}{\alpha -\alpha ^{1-\beta
^{-1}\left( \beta +\theta x\right) \exp \left( -\theta x\right) }}, \label{HazardRate2}
\end{equation}

\Bin $ x \in [lep(F), \ uep(F)]$ and for $x \in [lep(F_{1}), \ uep(F_{1})]$,

\begin{equation} 
h_{1}(x)=\frac{\theta \left( \beta -1+\theta x\right) }{\left( \beta +\theta x\right) 
}. \label{HazardRate1}
\end{equation}

\begin{figure}[htbp]
\begin{center}
\includegraphics[height=12cm, width=12cm]{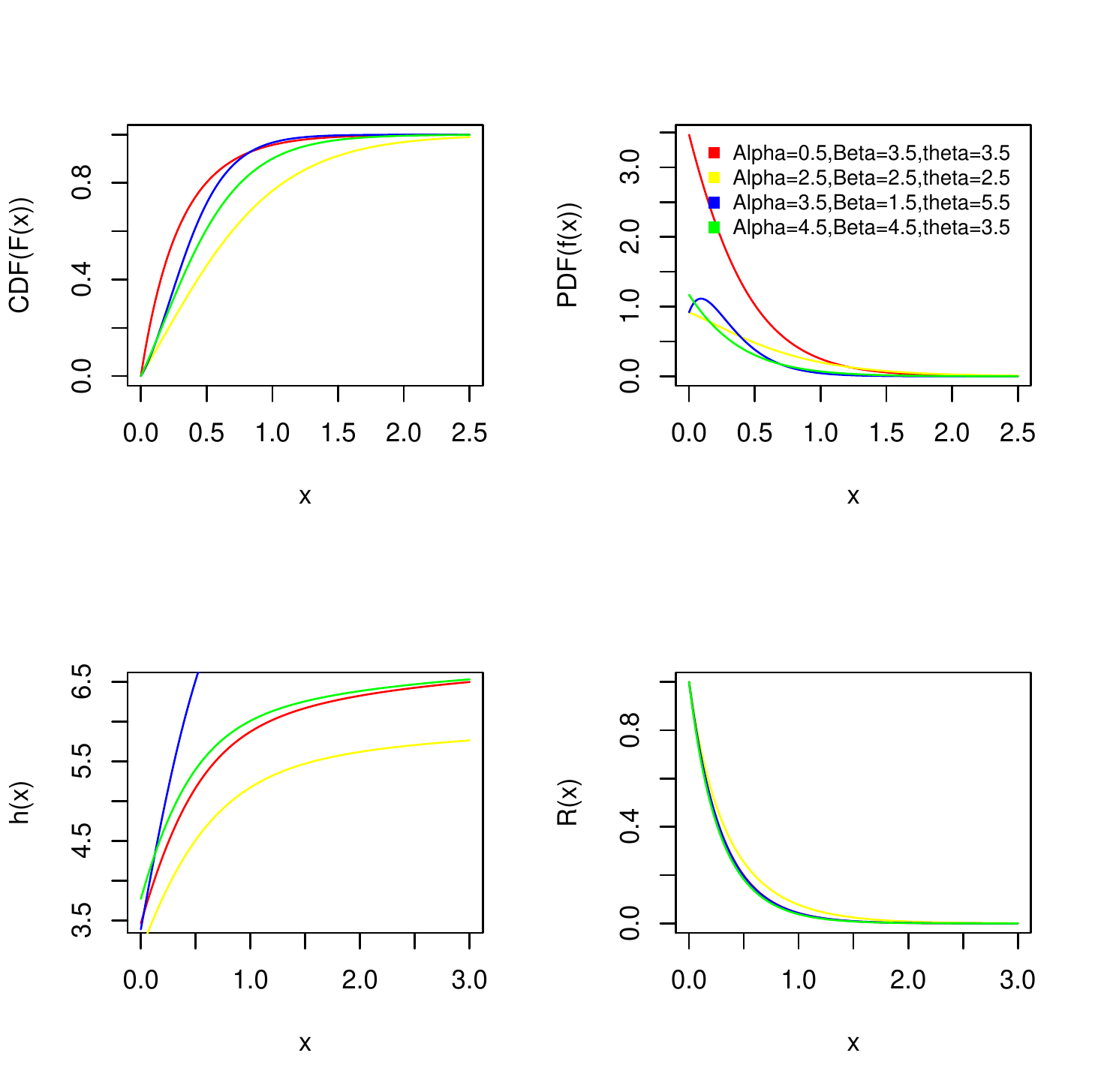}
\end{center}
\par
\caption{Graphs of \textit{cdf} (top left), \textit{pdf} (top right), hazard rate (bottom left) and Reliability (bottom right) of the \textit{PL-APT} with the same several values of parameters $\alpha$, $\beta$ and $\theta$.}
\label{figappld2}
\end{figure}

\subsection{Order statistics and Entropies}

\Bin Let $ X_{1},X_{2},...,X_{n}$ be a sample of random variables of size $ n $ 
following  a \textit{PL-APT} distribution with parameters $\alpha$, $\theta >0$
and $\beta >1 $ and  $ X_{\left( 1\right) },X_{\left( 2\right) },...,X_{\left(
n\right) } $ the order statistics of the processes. Then, the \textit{pdf} of the $%
k^{th}$ order statistic $X_{\left( k\right) }$ denoted by $g_{k}\left(
x\right) $ is defined as follows :

\begin{description}
\item[(1)] If $ \alpha \in ] 0, +\infty [\setminus \{1\} $ then 
\begin{equation}
g_{k}\left( x\right) =\left\{ 
\begin{array}{c}
\frac{n\left( n-1\right) !}{\left( n-k\right) !\left( k-1\right) !}\left( 
\frac{1-\alpha ^{1-\beta ^{-1}\left( \beta +\theta x\right) \exp \left(
-\theta x\right) }}{1-\alpha }\right) ^{k-1} \\ 
\\ 
\times \left( 1-\frac{1-\alpha ^{1-\beta ^{-1}\left( \beta +\theta x\right)
\exp \left( -\theta x\right) }}{1-\alpha }\right) ^{n-k} \\ 
\\ 
\times \frac{\theta \log \left( \alpha \right) \left( \beta -1+\theta
x\right) \exp \left( -\theta x\right) }{\beta \left( \alpha -1\right) }%
\alpha ^{1-\beta ^{-1}\left( \beta +\theta x\right) \exp \left( -\theta
x\right) } .%
\end{array}%
\right.  \label{denOrderstat1}
\end{equation}

\item[(2)] If $\alpha =1$ then 
\begin{equation}
g_{k}\left( x\right) =\left\{ 
\begin{array}{c}
\frac{n\left( n-1\right) !}{\left( n-k\right) !\left( k-1\right) !}\left(
1-\beta ^{-1}\left( \beta +\theta x\right) \exp \left( \theta x\right)
\right) ^{k-1} \\ 
\\ 
\times \left[ \beta ^{-1}\left( \beta +\theta x\right) \exp \left( \theta
x\right) \right] ^{n-k} \\ 
\\ 
\times \frac{\theta \left( \beta -1+\theta x\right) \exp \left( \theta
x\right) }{\beta }.%
\end{array}%
\right.  \label{denOrderstat22}
\end{equation}
\end{description}

\Bin We obtain the minimum and the maximum order statistics
respectively when $k=1$ and $k=n.$ If the size $n$ of the sample is an odd
number then there exists an integer number $m$ such that $n=2m+1$ and $k=m+1$
, and the distribution of the median is defined by :

\begin{description}
\item[(1)] If $\alpha \in ] 0, +\infty [\setminus \{1\}$ then 
\begin{equation*}
g_{m+1}\left( x\right) =\left\{ 
\begin{array}{c}
\frac{2m\left( 2m+1\right) !}{\left( m!\right) ^{2}}\left( \frac{1-\alpha
^{1-\beta ^{-1}\left( \beta +\theta x\right) \exp \left( -\theta x\right) }}{%
1-\alpha }\right) ^{m} \\ 
\\ 
\times \left( 1-\frac{1-\alpha ^{1-\beta ^{-1}\left( \beta +\theta x\right)
\exp \left( -\theta x\right) }}{1-\alpha }\right) ^{m} \\ 
\\ 
\times \frac{\theta \log \left( \alpha \right) \left( \beta -1+\theta
x\right) \exp \left( -\theta x\right) }{\beta \left( \alpha -1\right) }%
\alpha ^{1-\beta ^{-1}\left( \beta +\theta x\right) \exp \left( -\theta
x\right) } .%
\end{array}%
\right.
\end{equation*}

\item[(2)] If $\alpha =1$ then

\begin{equation*}
g_{m+1}\left( x\right) =\left\{ 
\begin{array}{c}
\frac{2m\left( 2m+1\right) !}{\left( m!\right) ^{2}}\left( 1-\beta
^{-1}\left( \beta +\theta x\right) \exp \left( -\theta x\right) \right) ^{m}
\\ 
\\ 
\times \left[ \beta ^{-1}\left( \beta +\theta x\right) \exp \left( -\theta
x\right) \right] ^{m} \\ 
\\ 
\times \frac{\theta \left( \beta -1+\theta x\right) \exp \left( -\theta
x\right) }{\beta }.%
\end{array}%
\right.
\end{equation*}
\end{description}

\subsection{Parameters estimation}

\Bin We estimate the parameters of the \textit{PL-APT} distribution by the maximum likelihood method. Let $X_{1}, X_{2}, . . . , X_{n} $ be a random sample from the \textit{PL-APT} distribution. Then, the log-likelihood function of the \textit{PL-APT} distribution $\emph{l}_{n}=\log \emph{l}\left( \beta ;\theta ;x_{1},x_{2},...,x_{n}\right) $ is given by

\begin{eqnarray}
\emph{l}_{n} &=&\sum_{i=1}^{n}\log \left( \frac{\theta \log \left( \alpha
\right) \left( \beta -1+\theta x_{i}\right) \exp \left( -\theta x_{i}\right) 
}{\beta \left( \alpha -1\right) }\alpha ^{1-\beta ^{-1}\left( \beta +\theta
x_{i}\right) e^{ \left( -\theta x_{i}\right)} }\right)  \label{loglikelihood}
\\
&=&n\left( \log \theta +\log \log \alpha -\log \beta -\log \left( \alpha
-1\right) \right) +\sum_{i=1}^{n}\log \left( \beta -1+\theta x_{i}\right) 
\notag \\
&&-\theta \sum_{i=1}^{n}x_{i}+\log \alpha \sum_{i=1}^{n}\left( 1-\beta
^{-1}\left( \beta +\theta x_{i}\right) \exp \left( -\theta x_{i}\right)
\right) .  \notag
\end{eqnarray}

\Bin The estimate values of $ \theta $ and $ \beta $, points in which the log-likelihood function attains its maximum, are
the solutions of likelihood equations (\ref{estimtetha}) and (\ref{estimbeta}) obtained by using the partial derivative for each parameter on
equation (\ref{loglikelihood}) and equating to zero. We have

\begin{eqnarray}
\frac{\partial \emph{l}_{n}}{\partial \theta } &=&\frac{n}{\theta }%
-\sum_{i=1}^{n}x_{i}+\sum_{i=1}^{n}\frac{x_{i}}{\beta -1+\theta x_{i}}%
-\left( 1+\frac{1}{\beta }\right) \log \alpha \sum_{i=1}^{n}x_{i} e^{-\theta x_{i}}  \label{estimtetha}
\end{eqnarray}
\begin{eqnarray*}
+\log \alpha \sum_{i=1}^{n}\left( \frac{\beta +\theta x_{i}}{\beta }%
\right) x_{i}e^{-\theta x_{i}} = 0,
\end{eqnarray*}

\Bin and

\begin{equation}
\frac{\partial \emph{l}_{n}}{\partial \beta }=-n\beta ^{-1}+\sum_{i=1}^{n}\frac{1}{%
\beta -1+\theta x_{i}}+\log \alpha \sum_{i=1}^{n}\beta ^{-2}\theta x_{i}\exp
\left( -\theta x_{i}\right) =0.  \label{estimbeta}
\end{equation}

\Bin The Likelihood equations (\ref{estimtetha}) and (\ref{estimbeta}) can not be solved explicitly since they are nonlinear functions of parameters $\theta$ and $\beta $. Therefore, iterative methods such as Newton-Raphson algorithm (NR) should be utilized to obtain the solution of these equations simultaneously.
 
\section{Asymptotic properties} \label{asymproperty}

\Bin In this part, we study the asymptotic properties specially the quantile function, the extremal quantile function and the extremal index estimation of the \textit{PL-APT} distribution.

\subsection{Quantile and extremal quantile functions} \label{quantextremquant}
\subsubsection{Quantile function}
\noindent The quantile function for the \textit{PL-APT} distribution is obtained by
solving for $x$ the non-linear equation $G_{\alpha}\left( x\right) =u$. The
quantile function of the \textit{PL-APT} distribution, for any
value of $\alpha >0$, as follows :

\noindent If $\alpha =1 $  then  

\begin{equation}\label{quantPL-APT1}
x\left( u\right) = -\frac{\beta }{\theta }-\frac{1}{\theta }W_{-1}\left[ \beta \left(
u-1\right) \exp \left( -\beta \right) \right],
\end{equation}

\noindent and if $\alpha \in ] 0, +\infty [\setminus \{1\} $ then

\begin{equation}\label{quantPL-APT} 
x\left( u\right) = -\frac{\beta }{\theta }-\frac{1}{\theta }W_{-1}\text{ }\left\{ -\beta \exp
\left( -\beta \right) +\frac{\beta \exp \left( -\beta \right) }{\log \alpha }
\log \left[ 1-\left( 1-\alpha \right) u\right] \right\}.
\end{equation}

\Bin  The proof of equation (\ref{quantPL-APT}) is presented in Appendix (A2), page \pageref{pageAppendixA2}.

\Bin  Obviously, we defined the first quartile, the median and the third quartile of the \textit{PL-APT} distribution by $Q1$, $Q2$ and $Q3$ respectively. For several values of the parameters, the Tables \ref{tab01} and \ref{tab02} comprise values of the quantile specially the first quartile, the median and the third quartile of the \textit{PL-APT} distribution.

\Bin For  $\alpha \in ] 0, +\infty [\setminus \{1\}$, we have the table \ref{tab01} and for $\alpha = 1$  we have the table  \ref{tab02} .

\begin{table}[h]
\centering
\begin{tabular}{c cc cc cc cc cc}
\hline \hline  
$ \theta$ && $\alpha$ &&$ \beta$ &&    $Q1$       &&  $Q2$    && $Q3$ \\ \hline 
          &&           &&         &&              &&          &&   \\
          && 0.5    &&1.1    && 0.7026004    && 1.2219140   && 1.7944430 \\    
    0.6   && 1.5    &&1.5    && 0.2253118    && 0.4127261   && 0.5760581 \\ 
          && 2      &&2.5    && 0.4515757    && 0.8640310   && 1.2577340  \\ 
				 &&           &&         &&              &&          &&   \\ \hline
        &&           &&         &&              &&          &&   \\
       && 0.5    &&1.1     && 0.28104020     && 0.4887654  && 0.7177772   \\    
   1.5 && 1.5    &&1.5     && 0.09012474    && 0.1650904  && 0.2304232   \\ 
       && 2      &&2.5     && 0.18063030    && 0.3456124  && 0.5030937    \\ 
			 &&           &&         &&              &&          &&   \\ \hline
       &&           &&         &&              &&          &&   \\

       && 0.5    &&1.1    && 0.14052010    && 0.24438270   && 0.3588886     \\    
     3 && 1.5    &&1.5    && 0.04506237    && 0.08254522   && 0.1152116    \\ 
       && 2      &&2.5    && 0.09031513    && 0.17280620   && 0.2515469    \\ 
			 &&           &&         &&              &&          &&   \\ \hline
       &&           &&         &&              &&          &&   \\
       && 0.5    &&1.1    && 0.08106928  && 0.14099000     && 0.20705110    \\    
   5.2 && 1.5    &&1.5    && 0.02599752  && 0.04762224     && 0.06646824     \\ 
       && 2      &&2.5    && 0.05210488  && 0.09969589     && 0.14512320     \\ 
				 &&           &&         &&              &&          &&   \\ \hline \hline

\end{tabular}
\caption{First Quartile (Q1), Median (Q2) and Third Quartile (Q3) For Selected Values of the Parameters of the \textit{PL-APT} distribution.}
\label{tab01}
\end{table} 


\begin{table}
\centering
\begin{tabular}{c cc cc cc cc cc}
\hline \hline  
$ \theta$ && $\alpha$ &&$ \beta$ &&    $Q1$       &&  $Q2$    && $Q3$ \\ \hline 
 &&           &&         &&              &&          &&   \\    
        && 1    &&1.1    && 1.4514270   && 2.643070   && 4.331719  \\    
    0.6 && 1    &&1.5    && 1.0760650   && 2.211457   && 3.869062  \\ 
        && 1    &&2.5    && 0.7581068   && 1.735562   && 3.277827  \\ 
				  &&           &&         &&              &&          &&   \\ \hline
        &&           &&         &&              &&          &&   \\

       && 1    &&1.1     && 0.5805708    && 1.0572280  && 1.732687   \\    
   1.5 && 1    &&1.5     && 0.4304260    && 0.8845827  && 1.547625    \\ 
       && 1    &&2.5     && 0.3032427    && 0.6942247  && 1.311131    \\ 
			 &&           &&         &&              &&          &&   \\ \hline
        &&           &&         &&              &&          &&   \\

       && 1    &&1.1    && 0.2902854    && 0.5286141    && 0.8663437    \\    
     3 && 1    &&1.5    && 0.2152130    && 0.4422913    && 0.7738124    \\ 
       && 1    &&2.5    && 0.1516214    && 0.3471123    && 0.6555655   \\ 
			  &&           &&         &&              &&          &&   \\ \hline
        &&           &&         &&              &&          &&   \\

       && 1    &&1.1    && 0.2902854   && 0.5286141   && 0.8663437   \\    
   5.2 && 1    &&1.5    && 0.1241614   && 0.2551681   && 0.4464302     \\ 
       && 1    &&2.5    && 0.0874738   && 0.2002571   && 0.3782109    \\ 
      &&       &&         &&              &&          &&   \\ \hline \hline 
        
\end{tabular}
\caption{First Quartile (Q1), Median (Q2) and Third Quartile (Q3) For Selected Values of the Parameters of the Pseudo-Lindley Distribution.}
\label{tab02}
\end{table} 

\Bin We can remark from Table \ref{tab01} and Table \ref{tab02} that :\\

\begin{enumerate}
	\item if $\alpha $ increases, $ \beta $ increases and  $ \theta $ constant then the value of each of $Q1$, $Q2$ and $Q3$ increases, 
	
	\item if $\alpha = \beta $ and $\theta$ increases then the value of each of $Q1$, $Q2$ and $Q3$ decreases. 
\end{enumerate}

\newpage
\subsubsection{Extremal quantile function}

\Bin To find the extremal quantile function of the \textit{PL-APT} distribution we solve the
equation $G_{\alpha}\left( x\right) =1-u$ , with $u\in \left( 0;1\right)$. If $\alpha
=1$ then $G_{\alpha}\left( x\right) =$ $F\left( x\right)$, is the \textit{cdf} of the Pseudo-Lindley
distribution. Its asymptotic properties are developed by \cite{gslo2020}. Now, we focus on the condition that $\alpha \in ] 0, +\infty [\setminus \{1\}$ and have the extremal quantile defined as follows\textcolor{red}{:}\\

\begin{eqnarray}
G_{\alpha}^{-1}\left( 1-u\right) &=&C_{0}+\theta ^{-1}\log \left( \frac{1}{u}\right)
+\theta ^{-1}\log \left( \log \left( \frac{1}{u}\right) \right) \notag \\
&+&\frac{\theta ^{-1}\log \left( -C\left( \alpha ,\beta \right) \right) }{\log \left( 
\frac{1}{u}\right) }+\theta ^{-1}K\left( u\right),  \label{QuAPPLD}
\end{eqnarray}

\Bin where $K\left( u\right) \longrightarrow 0 $  \ as \  $ u \longrightarrow 0 $. 

\Bin The details of the development of Equation (\ref{QuAPPLD}) are exposed in Appendix (A1), page \pageref{pageAppendixA1}.

\subsection{Extremes} \label{extremes}
\Ni In this part, we present the extremal properties of the \textit{PL-APT} distribution. We establish its domain of attraction, the expansion of the maximum values and finish with the study of the extremal value index estimation.

\subsubsection{Domain of attraction}
\Ni For any positive $\lambda $ , and $\alpha \in ] 0, +\infty [\setminus \{1\} $ we have the following limit
to determine the domain of attraction of the \textit{PL-APT} distribution, 

\begin{equation*}
\lim_{u\longrightarrow 0}\left( \frac{G_{\alpha}^{-1}\left( 1-\lambda u\right)
-G_{\alpha}^{-1}\left( 1-u\right) }{s\left( u\right) }\right) =L\left( \lambda
\right),
\end{equation*}%
\Ni where $$ s\left( u\right) = -u \left( G_{\alpha}^{-1}(1-u) \right)^{\prime}, \ \ \ 0 < u < 1. $$

\Ni We have 

\begin{eqnarray*}
L\left( \lambda \right) &=&\lim_{u\longrightarrow 0}\left( \frac{\theta
^{-1}\log \left( \frac{1}{\lambda }\right) +\frac{1}{\theta }\log \left( 1+%
\frac{\log 1/\lambda }{\log 1/u}\right) +\theta ^{-1}K\left( u\right) }{%
-u\left( G^{-1}\left( 1-u\right) \right) ^{\prime }}\right) \\
&& \\
&=&\lim_{u\longrightarrow 0}\left( \frac{\theta ^{-1}\log \left( \frac{1}{%
\lambda }\right) +\frac{1}{\theta }\log \left( 1+\frac{\log 1/\lambda }{\log
1/u}\right) +\theta ^{-1}K\left( u\right) }{-u\left( -\frac{1}{\theta }%
\times \frac{1}{u}\right) }\right) \\
&& \\
&=&\lim_{u\longrightarrow 0}\left( \log \left( \frac{1}{\lambda }\right)
+\log \left( 1+\frac{\log 1/\lambda }{\log 1/u}\right) +\theta ^{-1}K\left(
u\right) \right) \\
&=&-\log \lambda.
\end{eqnarray*}%

\Bin Since 

\begin{equation*}
\lim_{u\longrightarrow 0}\left( \frac{G_{\alpha}^{-1}\left( 1-\lambda u\right)
-G_{\alpha}^{-1}\left( 1-u\right) }{s\left( u\right) }\right) =-\log \left( \lambda
\right),
\end{equation*}%

\Bin by the $ \pi -$variation criteria developed in \cite{ips-wcia-ang} (see Proposition 11, page 88), we conclude that $G_{\alpha}$ belongs to the Gumbel domain denoted by  $ G_{\alpha}\in D\left( \Lambda \right) $.

\subsubsection{Expansion of the maximum values}

\bigskip Let $\ Z_{n}=-\log \left( nU_{1,n}\right) $ . By using the Renyi
representation , we have  that , if $\alpha \in ] 0, +\infty [\setminus \{1\}$ 
\begin{equation*}
M_{1}=X_{n,n}-G_{\alpha}^{-1}\left( 1-1/n\right) =G_{\alpha}^{-1}\left( 1-U_{1,n}\right)
-G_{\alpha}^{-1}\left( 1-1/n\right)
\end{equation*}%
\begin{eqnarray*}
M_{1} &=&C_{0}+\theta ^{-1}\log \left( \frac{1}{U_{1,n}}\right) +\theta
^{-1}\log \left( \log \left( \frac{1}{U_{1,n}}\right) \right) +\frac{\theta
^{-1}\log \left( -C\left( \alpha ,\beta \right) \right) }{\log \left( \frac{1%
}{U_{1,n}}\right) }+\theta ^{-1}K\left( U_{1,n}\right) \\
&-&C_{0}-\theta ^{-1}\log \left( n\right) -\theta ^{-1}\log \left( \log
\left( n\right) \right) -\frac{\theta ^{-1}\log \left( -C\left( \alpha
,\beta \right) \right) }{\log \left( n\right) }-\theta ^{-1}K\left(
1/n\right) . \\
&=&-\theta ^{-1}\log \left( nU_{1,n}\right) +\theta ^{-1}\log \left( 1+\frac{%
Z_{n}}{\log n}\right) -\theta ^{-1}\log \left( -C\left( \alpha ,\beta
\right) \right) \left( \frac{1}{\log \left( U_{1,n}\right) }+\frac{1}{\log
\left( n\right) }\right) \\
&+&\theta ^{-1}\left( K\left( U_{1,n}\right) -K\left( 1/n\right) \right),
\end{eqnarray*}

\Bin and hence 

\begin{eqnarray*}
&&\frac{X_{n,n}-G_{\alpha}^{-1}\left( 1-1/n\right) }{\left( 1/\theta \right) }\\
&&=-\log
\left( nU_{1,n}\right) +\log \left( 1+\frac{-\log \left( nU_{1,n}\right) }{%
\log n}\right) +\frac{-\log \left( nU_{1,n}\right) \log \left( -C\left(
\alpha ,\beta \right) \right) }{\left( \log n\right) \left( \log
U_{1,n}\right) } \\
&&+\left( K\left( U_{1,n}\right) -K\left( 1/n\right) \right) \\
&&=Z_{n}+\log \left( 1+\frac{Z_{n}}{\log n}\right) +\frac{Z_{n}\log \left(
-C\left( \alpha ,\beta \right) \right) }{\left( \log n\right) \left( \log
U_{1,n}\right) }+\left( K\left( U_{1,n}\right) -K\left( 1/n\right) \right).
\end{eqnarray*}

\Bin Since $\log \left( 1+\frac{Z_{n}}{\log n}\right) $ \ and $%
\frac{Z_{n}\log \left( -C\left( \alpha ,\beta \right) \right) }{\left( \log
n\right) \left( \log U_{1,n}\right) }$  converge both to $0$ as $%
u\longrightarrow  0 $ then we have

\begin{equation*}
\frac{X_{n,n}-G_{\alpha}^{-1}\left( 1-1/n\right) }{\left( 1/\theta \right) }=Z_{n}+O_{%
\mathbb{P}}\left( 1\right) .
\end{equation*}

\Bin So we have  $X_{n,n}$ converge to a Gumbel law $\Lambda $ with \textit{cdf}

\begin{equation*}
\Lambda (x)=\exp \left( -\exp \left( -x\right) \right) ,x\in \mathbb{R}.
\end{equation*}

\Ni Likewise, for $k=k\left( n\right) \longrightarrow +\infty 
$ such that $k\left( n\right) /n\longrightarrow 0$ , we have

\begin{eqnarray*}
M_{k} &=&\frac{X_{n-k,n}-G_{\alpha}^{-1}\left( 1-k/n\right) }{\left( 1/\theta \right) 
} \\
&=&\frac{G_{\alpha}^{-1}\left( 1-U_{k+1,n}\right) -G_{\alpha}^{-1}\left( 1-k/n\right) }{\left(
1/\theta \right) } \\
&=&\log \left( \frac{1}{U_{k+1,n}}\right) +\log \left( \log \left( \frac{1}{%
U_{k+1,n}}\right) \right) +\frac{\log \left( -C\left( \alpha ,\beta \right)
\right) }{\log \left( \frac{1}{U_{k+1,n}}\right) }+K\left( U_{k+1,n}\right)
\\
&&-\log \left( n/k\right) -\log \left( \log \left( n/k\right) \right) -\frac{%
\log \left( -C\left( \alpha ,\beta \right) \right) }{\log \left( n/k\right) }%
-K\left( k/n\right). 
\end{eqnarray*}

\Bin Hence,

\begin{eqnarray*}
M_{k} &=&-\log \left( nU_{k+1,n}/k\right) +\log \left( 1+\frac{-\log \left(
nU_{k+1,n}/k\right) }{\log \left( n/k\right) }\right) \\ &+& \log \left( -C\left(
\alpha ,\beta \right) \right) \left( \frac{1}{\log \left( \frac{1}{U_{k+1,n}}%
\right) }+\frac{1}{\log \left( n/k\right) }\right)+K\left( U_{k+1,n}\right) -K\left( k/n\right) \\
&=&T_{n}+\log \left( 1+\frac{T_{n}}{\log q_{n}}\right) +\log \left( -C\left(
\alpha ,\beta \right) \right) \left( \frac{\log \left( kU_{k+1,n}/n\right) }{%
\log \left( U_{k+1,n}\right) \log \left( n/k\right) }\right) +O_{\mathbb{P}%
}\left( \left( \log q_{n}\right) ^{-2}\right),
\end{eqnarray*}

\Bin where $T_{n}=-\log \left( nU_{k+1,n}/k\right) $ and $%
q_{n}=n/k$ which goes to $\ +\infty $ as $n\longrightarrow +\infty .$ So ,
we have

\begin{eqnarray*}
\frac{X_{n-k,n}-G_{\alpha}^{-1}\left( 1-k/n\right) }{\left( 1/\theta \right) }
&=&T_{n}+\log \left( 1+\frac{T_{n}}{\log q_{n}}\right) \\
&+&\log \left( -C\left( \alpha ,\beta \right) \right) \left( \frac{\log
\left( kU_{k+1,n}/n\right) }{\log \left( U_{k+1,n}\right) \log \left(
n/k\right) }\right) + O_{\mathbb{P}}\left( \left( \log q_{n}\right)
^{-2}\right) .
\end{eqnarray*}

\subsubsection{Estimation of the extreme value index}

\Bin In \cite{ngomlo2016} create a new class of estimators of
the extreme value index  built around the statistic defined by

\begin{equation*}
T_{n}\left( f,s\right) =\sum_{j=1}^{k(n)}f\left( j\right) \left[ \log
X_{n-j+1,n}-\log X_{n-j,n}\right] ^{s},
\end{equation*}

\Bin where $f$ is a positive measurable mapping defined from $ \mathbb{N} -\left\{ 0\right\} $ to $ \mathbb{R}-\left\{ 0\right\} $, and $s$ is a positive real
number. To estimate the extreme value index , it is necessary to define the
following expressions. We have

\begin{equation*}
a_{n}\left( f,s\right) =\Gamma \left( s+1\right) \sum_{j=1}^{k(n)}f\left(
j\right) j^{-s},
\end{equation*}

\begin{equation*}
s_{n}^{2}\left( f,s\right) =\left\{ \Gamma \left( 2s+1\right) -\Gamma
^{2}\left( s+1\right) \right\} \sum_{j=1}^{k(n)}f^{2}\left( j\right) j^{-2s}
\end{equation*}

\Bin and

\begin{equation*}
B_{n}\left( f,s\right) =\max \left\{ \frac{f\left( j\right) j^{-s}}{%
s_{n}\left( f,s\right) },1\leq j\leq k\right\} .
\end{equation*}

\Bin The \cite{ngomlo2016} estimator, called the double Hill
estimator is defined by the expression below

\begin{equation*}
M_{n}\left( f,s\right) =\left( \frac{T_{n}\left( f,s\right) }{a_{n}\left(
f,s\right) }\right) ^{1/s}.
\end{equation*}

\Bin We remark that, if $ f\left( j\right) =j $ and $s=1$ then
\bigskip\ $M_{n}\left( f,s\right) =M_{n}\left( j,1\right) =H_{n}$ is the
\cite {hill1975} estimator and if $ f\left( j\right) =j^{\tau };s>0$ and $s=1 $ then
\bigskip\ $M_{n}\left( f,s\right) =M_{n}\left( j^{\tau },1\right)
=T_{n}\left( \tau ,1\right) $ is the \cite{demedioplo} estimator .

\begin{theorem}

\Bin We have \\

\Bin (a) If the following conditions $ \frac{a_{n}\left( f,s\right) }{s_{n}\left( f,s\right) }\rightarrow 0$ and $
B_{n}\left( f,s\right) \rightarrow 0$  hold as $n\rightarrow +\infty ,$ then 
\begin{equation*}
\frac{a_{n}\left( f,s\right) }{s_{n}\left( f,s\right) }\left[ M_{n}\left(
f,s\right) -\gamma ^{s}\right] \rightsquigarrow \mathcal{N} \left( 0,\gamma
^{2s}\right) .
\end{equation*}

\Bin (b) Furthermore, if  

$$ \frac{a_{n}\left( f,s\right) }{s_{n}\left( f,s\right) }\rightarrow +\infty $$ 

 \Bin then

\begin{equation*}
\frac{a_{n}\left( f,s\right) }{s_{n}\left( f,s\right) }\left[ M_{n}\left(
f,s\right) -\gamma ^{s}\right] \rightsquigarrow \mathcal{N} \left( 0,\gamma
^{2s}C^{2}\left( s\right) \right) .
\end{equation*}
 \label{doubldhilltheo}
\end{theorem}

\begin{proof}
\Bin Here, we establish the proof of Theorem \ref{doubldhilltheo}.

\Bin As the \textit{cdf}  $G_{\alpha} $ of the \textit{PL-APT}  belongs  to the attraction  domain of Gumbel, it is known that to estimate the extremal index it is equivalent to use in the calculus the $ G_{\alpha}^{-1}(1-u) $ or $\log G_{\alpha}^{-1}(1-u) $.

\Bin Let
\begin{eqnarray*}
R_{j,n} &=&f(j)\left( \log X_{n-j+1,n}-\log X_{n-j,n}\right) ^{s} \\
&=&f(j)\left(  G_{\alpha}^{-1}\left( 1-U_{j,n}\right) - G_{\alpha}^{-1}\left(
1-U_{j+1,n}\right) \right) ^{s} \\
&=&f(j)\left( (1/\theta) \log \left( \frac{U_{j+1,n}}{U_{j,n}}\right) +C_{2}A_{n}+O_{%
\mathbb{P}}\left( B_{n}\right) \right) ^{s} \\
&=&f(j)\left( (1/\theta)j^{-1}E_{j,n}+C_{2}A_{n}+O_{\mathbb{P}}\left( B_{n}\right)
\right) ^{s}
\end{eqnarray*}

\bigskip \noindent with

\begin{equation*}
A_{n}=\max \left\{ U_{j,n}^{2};U_{j+1,n}^{2}\right\} \leq -2\log \left( 
\frac{U_{j+1,n}}{U_{j,n}}\right)
\end{equation*}

\bigskip \noindent and

\begin{equation*}
O_{\mathbb{P}}\left( B_{n}\right) =O_{\mathbb{P}}\left( \left( \log n\right)
^{-2}\right).
\end{equation*}

\bigskip \noindent By the mean value theorem and $j\in \left\{
1,...,k\left( n\right) \right\} ,$ $s>1,$ we get

\begin{eqnarray*}
 &  & \\
 &  & R_{j,n}-(1/\theta)^{s}f(j)j^{-s}E_{j,n}^{s} \\
 &\leq &sf(j)\left\vert C_{2}A_{n}+O_{%
\mathbb{P}}\left( B_{n}\right) \right\vert \left( (1/\theta)j^{-1}E_{j,n}+\left\vert
C_{2}A_{n}\right\vert +\left\vert O_{\mathbb{P}}\left( B_{n}\right)
\right\vert \right) ^{s-1} \\
&\leq & (1/\theta)sf(j)j^{-1}E_{j,n}\left( (1/\theta)j^{-1}E_{j,n}+\left\vert
C_{2}A_{n}\right\vert +\left\vert O_{\mathbb{P}}\left( B_{n}\right)
\right\vert \right) ^{s-1} \\
&\leq &(1/\theta)sf(j)j^{-1}E_{j,n}\left(
D_{s}(1/\theta)^{s-1}j^{s-1}E_{j,n}^{s-1}+(1/\theta)D_{s}^{2}j^{-1}E_{j,n}+O_{\mathbb{P}%
}\left( \frac{D_{s}^{2}}{\left( \log n\right) ^{-2}}\right) \right) .
\end{eqnarray*}

\bigskip \noindent Applying the sum , we get

\begin{eqnarray*}
 &  & \\
 &  &  \sum_{j=1}^{k\left( n\right) }\left(
R_{j,n}-(1/\theta)^{s}f(j)j^{-s}E_{j,n}^{s}\right) \\
 &\leq &(1/\theta)s\sum_{j=1}^{k\left(
n\right) }f(j)j^{-1}E_{j,n}\left(
D_{s}(1/\theta)^{s-1}j^{s-1}E_{j,n}^{s-1}+(1/\theta)D_{s}^{2}j^{-1}E_{j,n}+O_{\mathbb{P}%
}\left( \frac{D_{s}^{2}}{\left( \log n\right) ^{-2}}\right) \right) \\
\end{eqnarray*}

\Bin and

\begin{eqnarray*}
&  & \\
 &  & \sum_{j=1}^{k\left( n\right) }R_{j,n}-(1/\theta)^{s}\sum_{j=1}^{k\left( n\right)
}f(j)j^{-s}E_{j,n}^{s} \\
&\leq & (1/\theta)s\sum_{j=1}^{k\left( n\right)
}f(j)j^{-1}E_{j,n}\left(
D_{s}(1/\theta)^{s-1}j^{s-1}E_{j,n}^{s-1}+(1/\theta)D_{s}^{2}j^{-1}E_{j,n}+O_{\mathbb{P}%
}\left( \frac{D_{s}^{2}}{\left( \log n\right) ^{-2}}\right) \right) \\
\end{eqnarray*}

\begin{eqnarray*}
 &  & \left\vert T_{n}\left( f,s\right) -(1/\theta)^{s}S_{n}\left( f,s\right) \right\vert \\
&\leq &(1/\theta)sS_{n}\left( f,1\right) \left( D_{s}(1/\theta)^{s-1}S_{n}\left( f,s-1\right)
+(1/\theta)D_{s}^{2}S_{n}\left( f,1\right) +O_{\mathbb{P}}\left( \frac{D_{s}^{2}}{%
\left( \log n\right) ^{-2}}\right) \right) \text{ (IL)}
\end{eqnarray*}

\Bin where

\begin{equation*}
S_{n}\left( f,s\right) =\sum_{j=1}^{k\left( n\right) }f(j)j^{-s}E_{j,n}^{s}
\end{equation*}

\Bin and

\begin{equation*}
T_{n}\left( f,s\right) =\sum_{j=1}^{k\left( n\right) }f(j)\left( \log
X_{n-j+1,n}-\log X_{n-j,n}\right) ^{s}.
\end{equation*}

\bigskip \noindent The random variable $S_{n}\left( f,s\right) $ is a
sequence of partial sum of random real values and independent random variables
indexed by $j\in \left\{ 1,...,k\left( n\right) \right\} $ with first and
second moments

\begin{equation*}
\mu _{1}=\Gamma \left( s+1\right) f\left( j\right) j^{-s}\text{ \ and \ }\mu
_{2}=\left( \text{\ }\Gamma \left( 2s+1\right) -\Gamma \left( s+1\right)
^{2}\right) f\left( j\right) j^{-s}.
\end{equation*}

\Bin Its asymptotic normality is given as follows by using the
\ Levy-Feller-Linderberg (see \cite{ips-mfpt-ang}, Theorem 20),

\begin{equation*}
\left( \frac{1}{s_{n}\left( f,s\right) }\sum_{j=1}^{k\left( n\right) }\left(
f(j)j^{-s}\left( E_{j,n}^{s}-\text{\ }\Gamma \left( s+1\right) \right)
\right) \leadsto \mathcal{N}\left( 0,1\right) \right) 
\end{equation*}

\Bin $\text{ \ and }
B_{n}\left( f,s\right) \longrightarrow 0\text{ as }n\longrightarrow +\infty $,

\bigskip \noindent where

\begin{equation*}
B_{n}\left( f,s\right) =\frac{1}{C\left( s\right) }\left\{ \frac{\mathbb{V}%
ar\left( f(j)j^{-s}\left( E_{j,n}^{s}-\text{\ }\Gamma \left( s+1\right)
\right) \right) }{\sum_{j=1}^{k\left( n\right) }\mathbb{V}ar\left(
f(j)j^{-s}\left( E_{j,n}^{s}-\text{\ }\Gamma \left( s+1\right) \right)
\right) }\right\} .
\end{equation*}

\Bin By combining the Lindeberg condition, the Cauchy-Schwarz
inequality and the central limit theorem, we get for $S_{n}\left(
f,s\right) $ the result below

\begin{equation*}
\frac{S_{n}\left( f,s\right) -(1/\theta)^{s}a_{n}\left( f,s\right) }{s_{n}\left(
f,s\right) }\leadsto \mathcal{N}\left( 0,1\right) .
\end{equation*}

\Bin The continuation of the inequality (IL) implies

\begin{eqnarray*}
\frac{T_{n}\left( f,s\right) -(1/\theta)^{s}S_{n}\left( f,s\right) }{s_{n}\left(
f,s\right) }-\frac{S_{n}\left( f,s\right) -(1/\theta)^{s}a_{n}\left( f,s\right) }{%
s_{n}\left( f,s\right) } &\leq &O_{\mathbb{P}}\left( \frac{S_{n}\left(
f,1\right) }{S_{n}\left( f,s\right) \log n}\right) \\
\end{eqnarray*}

\begin{eqnarray*}\label{equIL1}
\frac{T_{n}\left( f,s\right) -S_{n}\left( f,s\right) }{s_{n}\left(
f,s\right) }-\frac{(1/\theta)^{s}\left( S_{n}\left( f,s\right) -a_{n}\left(
f,s\right) \right) }{s_{n}\left( f,s\right) } &\leq &O_{\mathbb{P}}\left( 
\frac{S_{n}\left( f,1\right) }{S_{n}\left( f,s\right) \log n}\right). (IL1)
\end{eqnarray*}
\Bin The right hand side of the inequality $(IL1)$ above tends to zero in probability if and only if 
\begin{eqnarray*}
\left( \frac{S_{n}\left( f,1\right) }{S_{n}\left( f,s\right) \log n}\right) \longrightarrow 0 \  \  \ as \  \ \ n \longrightarrow +\infty .
\end{eqnarray*}

\Bin Let 
\begin{equation*}
M_{n}\left( f,s\right) =\left( \frac{T_{n}\left( f,s\right) }{a_{n}\left(
f,s\right) }\right) ^{1/s}.
\end{equation*}

\bigskip \noindent By combining the results above for the left hand side of the inequality $(IL1)$ above, we arrive at

\begin{equation*}
\frac{T_{n}\left( f,s\right) -(1/\theta)^{s}S_{n}\left( f,s\right) }{s_{n}\left(
f,s\right) }=\frac{a_{n}\left( f,s\right) }{s_{n}\left( f,s\right) }\left[ 
\frac{T_{n}\left( f,s\right) }{a_{n}\left( f,s\right) }-(1/\theta)^{s}\right].
\end{equation*}
\Bin So, we have 
\begin{equation*}
\frac{a_{n}\left( f,s\right) }{s_{n}\left( f,s\right) }\left[ 
\frac{T_{n}\left( f,s\right) }{a_{n}\left( f,s\right) }-(1/\theta)^{s}\right] = Z_{n}+O_{\mathbb{P}}(1).
\end{equation*}

\Bin If $ a_{n}\left( f,s\right) / s_{n}\left( f,s\right)  \longrightarrow +\infty $, then we have by using the delta method applied to $g(t)=t^{1/s} $ , 

\begin{equation*}
\frac{a_{n}\left( f,s\right) }{s_{n}\left( f,s\right) }\left[ 
\frac{T_{n}\left( f,s\right) }{a_{n}\left( f,s\right) }-(1/\theta)^{s}\right] \rightsquigarrow \mathcal{N} \left(0,(1/\theta)^{2s}C^{2}(s) \right).
\end{equation*}
\end{proof}

\section{Conclusion} \label{conclud}
\Bin In this paper, the \textit{PL-APT} distribution, which is flexible for modeling lifetime data, is presented. This study is motivated by the extensive use of the Pseudo-Lindley distribution in Statistics and  Economics. The \textit{PL-APT} distribution provides more flexibility than the Lindley and the Pseudo-Lindley distributions to analyze lifetime data. The \textit{PL-APT} distribution has several new and known properties as its mathematical properties and asymptotic convergence of the extreme value index. Its parameters are estimated  by the maximum likelihood method. The possibility of expanding Pseudo-Lindley into other areas can be achieved with the new quantile distribution and the extremal quantile distribution of the \textit{PL-APT}. In a next article, we face to study the simulations with application of real lifetime data.\\

\Bin \textbf{Acknowledgment.} The authors Ngom,  Diallo and Fall thank Professor Lo for his advice, invaluable help, comments and recommendations.\\


\newpage

\Bin \textbf{Appendix (A1): Extremal quantile of the \textit{PL-APT}  distribution.\label{pageAppendixA1}} \\

\Bin We solve for $x$ the equation $G_{\alpha}\left( x\right) =1-u$ , $\alpha \in ] 0, +\infty [\setminus \{1\}$. Thus, we have 
\begin{eqnarray*}
\frac{1-\alpha ^{1-\beta ^{-1}\left( \beta +\theta x\right) \exp \left(
-\theta x\right) }}{1-\alpha } &=&1-u \\
&& \\
1-\alpha ^{1-\beta ^{-1}\left( \beta +\theta x\right) \exp \left( -\theta
x\right) } &=&\left( 1-u\right) \left( 1-\alpha \right) \\
&& \\
\alpha ^{1-\beta ^{-1}\left( \beta +\theta x\right) \exp \left( -\theta
x\right) } &=&\alpha +u\left( 1-\alpha \right) \\
&& \\
1-\beta ^{-1}\left( \beta +\theta x\right) \exp \left( -\theta x\right)
&=&\log \left( \alpha +u\left( 1-\alpha \right) \right) \times \left( \log
\alpha \right) ^{-1} \\
&& \\
-\beta ^{-1}\left( \beta +\theta x\right) \exp \left( -\theta x\right)
&=&-1+\log \left( \alpha +u\left( 1-\alpha \right) \right) \times \left(
\log \alpha \right) ^{-1} \\
&& \\
-\left( \beta +\theta x\right) \exp \left( -\theta x\right) &=&-\beta +\frac{%
\beta }{\log \alpha }\log \left( \alpha +u\left( 1-\alpha \right) \right).
\\
&&
\end{eqnarray*}

\Bin Multiplying both sides by $\exp \left( -\beta \right),$ we have

\begin{eqnarray*}
-\left( \beta +\theta x\right) \exp \left( -\theta x\right) \exp \left(
-\beta \right) &=&\exp \left( -\beta \right) \left( -\beta +\frac{\beta }{%
\log \alpha }\log \left( \alpha +u\left( 1-\alpha \right) \right) \right) \\
\end{eqnarray*}

\begin{equation}\label{equEL}
\left( -\beta -\theta x\right) \exp \left( -\beta -\theta x\right) =-\beta
\exp \left( -\beta \right) +\frac{\beta \exp \left( -\beta \right) }{\log
\alpha }\log \left( \alpha +u\left( 1-\alpha \right) \right) . 
\end{equation}

\Bin The right hand side  of the equation (\ref{equEL}) satisfies the negative branch $W_{-1}\left( .\right) $ of the Lambert $W$ function. Let $W\left( x\right) = - \beta -\theta x$. Herein, we have

\begin{eqnarray*}
W(x)\exp \left( W(x)\right) &=&-\beta \exp \left( -\beta \right) +\frac{%
\beta \exp \left( -\beta \right) }{\log \alpha }\log \left( \alpha +u\left(
1-\alpha \right) \right) , \\
&&
\end{eqnarray*}

\Bin and

\begin{equation*}
-\beta -\theta x=W_{-1}\left( -\beta \exp \left( -\beta \right) +\frac{\beta
\exp \left( -\beta \right) }{\log \alpha }\log \left( \alpha +u\left(
1-\alpha \right) \right) \right).
\end{equation*}

\Bin Hence,

\begin{eqnarray*}
-\beta -\theta x &=&W_{-1}\left( -\beta \exp \left( -\beta \right) +\frac{%
\beta \exp \left( -\beta \right) }{\log \alpha }\log \left( \alpha +u\left(
1-\alpha \right) \right) \right) \\
&& \\
-\theta x &=&\beta +W_{-1}\left( -\beta \exp \left( -\beta \right) +\frac{%
\beta \exp \left( -\beta \right) }{\log \alpha }\log \left( \alpha +u\left(
1-\alpha \right) \right) \right) \\
&& \\
x &=&-\frac{\beta }{\theta }-\frac{1}{\theta }W_{-1}\left( -\beta \exp
\left( -\beta \right) +\frac{\beta \exp \left( -\beta \right) }{\log \alpha }%
\log \left( \alpha +u\left( 1-\alpha \right) \right) \right) \\
&& \\
x &=&-\frac{\beta }{\theta }-\frac{1}{\theta }W_{-1}\left( A\left( \alpha
;\beta ;u\right) \right),
\end{eqnarray*}

\Bin where 

\begin{equation}
\left( A\left( \alpha ,\beta ,u\right) \right) = -\beta \exp
\left( -\beta \right) +\frac{\beta \exp \left( -\beta \right) }{\log \alpha }%
\log \left( \alpha +u\left( 1-\alpha \right) \right).
\end{equation}

\Bin Now we deal with the expression \ \ $A\left( \alpha
;\beta ;u\right) $ below

\begin{eqnarray*}
A\left( \alpha ;\beta ;u\right) &=&-\beta \exp \left( -\beta \right) +\frac{%
\beta \exp \left( -\beta \right) }{\log \alpha }\log \left( \alpha +u\left(
1-\alpha \right) \right) \\
&=&-\beta \exp \left( -\beta \right) +\frac{\beta \exp \left( -\beta \right) 
}{\log \alpha }\left( \log \alpha +\log \left( 1+u\left( \frac{1-\alpha }{%
\alpha }\right) \right) \right) \\
&=&\frac{\beta \exp \left( -\beta \right) }{\log \alpha }\log \left(
1+u\left( \frac{1-\alpha }{\alpha }\right) \right) \\
&=&\frac{\beta \exp \left( -\beta \right) }{\log \alpha }\left[ u\left( 
\frac{1-\alpha }{\alpha }\right) +O\left( u^{2}\right) \right] \\
&=&u\left( \frac{1-\alpha }{\alpha }\right) \left( \frac{\beta \exp \left(
-\beta \right) }{\log \alpha }\right) +O\left( u^{2}\right) \\
&=&uC\left( \alpha ,\beta \right) +O\left( u^{2}\right),
\end{eqnarray*}

\Ni where

\begin{equation*}
C\left( \alpha ,\beta \right) =\frac{\left( 1-\alpha \right) \beta \exp
\left( -\beta \right) }{\alpha \log \alpha }, 
\end{equation*}%

$C\left( \alpha ,\beta \right) $ is a negative \ number for $\alpha \in ] 0, +\infty [\setminus \{1\}$ and $\beta >0.$

\Bin Now, the inverse Lambert function has the following
expansion at the neighborhood of zero $(x=0^{-})$ :

\begin{equation*}
W_{-1}\left( x\right) =\log \left( -x\right) -\log \left( -\log \left(
-x\right) \right) +O \left( \frac{\log \left( -\log \left( -x\right)
\right) }{\log \left( -x\right) }\right)  .
\end{equation*}

\Bin By applying  the $A\left( \alpha ;\beta ;u\right) $ on it, we have

\begin{eqnarray*}
W_{-1}\left( A\left( \alpha ;\beta ;u\right) \right) &=&\log \left( -A\left(
\alpha ;\beta ;u\right) \right) -\log \left( -\log \left( -A\left( \alpha
;\beta ;u\right) \right) \right) \\
&& + O \left( \frac{\log \left( -\log \left( -A\left( \alpha ;\beta
;u\right) \right) \right) }{\log \left( -A\left( \alpha ;\beta ;u\right)
\right) }\right)  . 
\end{eqnarray*}%

\Bin Since the expansion of \ \ $A\left( \alpha ;\beta ;u\right) $ gives \ $%
A\left( \alpha ;\beta ;u\right) =uC\left( \alpha ,\beta \right) +O\left(
u^{2}\right) $ , we have

\begin{equation*}
-A\left( \alpha ;\beta ;u\right) =-uC\left( \alpha ,\beta \right) +O\left(
u^{2}\right). \\
\end{equation*}

\Bin Applying the logarithm both sides, we have

\begin{eqnarray*}
\log \left( -A\left( \alpha ;\beta ;u\right) \right) &=&\log \left[
-uC\left( \alpha ,\beta \right) +O\left( u^{2}\right) \right] \\
&& \\
&=&\log \left[ -uC\left( \alpha ,\beta \right) \left( 1+O\left( u\right)
\right) \right] \\
&& \\
&=&\log \left( u\right) +\log \left( -C\left( \alpha ,\beta \right) \right)
+O\left( u\right) . \\
&\rightarrow &-\infty.
\end{eqnarray*}%

\Bin And, by re-applying the logarithm both sides, we have

\begin{eqnarray*}
-\log \left( -A\left( \alpha ;\beta ;u\right) \right) &=&-\log \left(
u\right) -\log \left( -C\left( \alpha ,\beta \right) \right) +O\left(
u\right) \\
&& \\
-\log \left[ -\log \left( -A\left( \alpha ;\beta ;u\right) \right) \right]
&=&-\log \left[ -\log \left( u\right) -\log \left( -C\left( \alpha ,\beta
\right) \right) +O\left( u\right) \right] \\
&& \\
&=&-\log \left[ -\log \left( u\right) \left( 1+\frac{\log \left( -C\left(
\alpha ,\beta \right) \right) }{\log \left( u\right) }+O\left( \frac{-u}{%
\log \left( u\right) }\right) \right) \right] \\
&& \\
&=&-\log \left( -\log \left( u\right) \right) -\log \left( 1+\frac{\log
\left( -C\left( \alpha ,\beta \right) \right) }{\log \left( u\right) }%
+O\left( \frac{-u}{\log \left( u\right) }\right) \right) \\
&& \\
&=&-\log \left( \log \left( \frac{1}{u}\right) \right) +\frac{\log \left(
-C\left( \alpha ,\beta \right) \right) }{\log \left( \frac{1}{u}\right) }%
+O\left( \frac{u}{\log \left( \frac{1}{u}\right) }\right) . \\
&\rightarrow &-\infty.
\end{eqnarray*}%

\Bin Hence,
 
\begin{equation*}
\frac{\log \left( -\log \left( -A\left( \alpha ;\beta ;u\right) \right)
\right) }{\log \left( -A\left( \alpha ;\beta ;u\right) \right) }=\frac{-\log
\left( \log \left( \frac{1}{u}\right) \right) +\frac{\log \left( -C\left(
\alpha ,\beta \right) \right) }{\log \left( \frac{1}{u}\right) }+O\left( 
\frac{u}{\log \left( \frac{1}{u}\right) }\right) }{\log \left( u\right)
+\log \left( -C\left( \alpha ,\beta \right) \right) +O\left( u\right) }
\end{equation*}

\begin{eqnarray*}
&=&\frac{-\log \left( \log \left( \frac{1}{u}\right) \right) \left[ 1+\frac{%
\log \left( -C\left( \alpha ,\beta \right) \right) }{-\log \left( \log
\left( \frac{1}{u}\right) \right) \log \left( \frac{1}{u}\right) }+O\left( 
\frac{u}{-\log \left( \log \left( \frac{1}{u}\right) \right) \log \left( 
\frac{1}{u}\right) }\right) \right] }{\log \left( u\right) \left[ 1+\frac{%
\log \left( -C\left( \alpha ,\beta \right) \right) }{\log \left( u\right) }%
+O\left( \frac{u}{\log \left( u\right) }\right) \right] } \\
&=&\frac{\log \left( \log \left( \frac{1}{u}\right) \right) }{\log \left( 
\frac{1}{u}\right) }\left[ 1+O\left( \frac{u}{\log \left( u\right) }\right) %
\right],
\end{eqnarray*}

\Bin which tends to $0$ as $u\rightarrow 0.$

\Bin Therefore,

\begin{eqnarray*}
W_{-1}\left( A\left( \alpha ;\beta ;u\right) \right) &=&\log \left( u\right)
-\log \left( \log \left( \frac{1}{u}\right) \right) \\
&&+\log \left( -C\left( \alpha ,\beta \right) \right) +\frac{\log C\left(
\alpha ,\beta \right) }{\log \left( \frac{1}{u}\right) }+O\left( \frac{u}{%
\log \left( \frac{1}{u}\right) }\right). 
\end{eqnarray*}
\Ni Hence,

\begin{eqnarray*}
x &=&-\frac{\beta }{\theta }-\frac{1}{\theta }\left[ W_{-1}\left( A\left(
\alpha ;\beta ;u\right) \right) \right] \\
&=&-\frac{\beta }{\theta }+\frac{1}{\theta }\log \left( \frac{1}{u}\right) +%
\frac{1}{\theta }\log \left( \log \left( \frac{1}{u}\right) \right) -\frac{1%
}{\theta }\log \left( -C\left( \alpha ,\beta \right) \right) \\
&&+\frac{1}{\theta }\times \frac{\log \left( -C\left( \alpha ,\beta \right)
\right) }{\log \left( \frac{1}{u}\right) }+O\left( \frac{u}{\log \left( 
\frac{1}{u}\right) }\right) \\
&=&-\beta \theta ^{-1}-\theta ^{-1}\log \left( -C\left( \alpha ,\beta
\right) \right) +\theta ^{-1}\log \left( \frac{1}{u}\right) +\theta
^{-1}\log \left( \log \left( \frac{1}{u}\right) \right) \\
&&+\frac{\theta ^{-1}\log \left( -C\left( \alpha ,\beta \right) \right) }{%
\log \left( \frac{1}{u}\right) }+\theta ^{-1}K\left( u\right) \\
&=&C_{0}+\theta ^{-1}\log \left( \frac{1}{u}\right) +\theta ^{-1}\log \left(
\log \left( \frac{1}{u}\right) \right) +\frac{\theta ^{-1}\log \left(
-C\left( \alpha ,\beta \right) \right) }{\log \left( \frac{1}{u}\right) }%
+\theta ^{-1}K\left( u\right) .
\end{eqnarray*}

\Ni Where $K\left( u\right) =O\left( \frac{u}{\log \left( \frac{1}{u}%
\right) }\right) $\noindent and \ $C_{0}=-\beta \theta ^{-1}-\theta
^{-1}\log \left( -C\left( \alpha ,\beta \right) \right) .$ 

\Bin Therefore, if $\alpha \in ] 0, +\infty [\setminus \{1\}$, we have

\begin{equation}
G_{\alpha}^{-1}\left( 1-u\right) =C_{0}+\theta ^{-1}\log \left( \frac{1}{u}\right)
+\theta ^{-1}\log \left( \log \left( \frac{1}{u}\right) \right) +\frac{%
\theta ^{-1}\log \left( -C\left( \alpha ,\beta \right) \right) }{\log \left( 
\frac{1}{u}\right) }+\theta ^{-1}K\left( u\right) . 
\end{equation}

\Ni \textbf {Appendix (A2) : Quantile of the \textit{PL-APT} distribution. \label{pageAppendixA2}}  \\

\Ni For $\alpha =1$, the quantile function of the \textit{PL-APTD} is the quantile of Pseudo-Lindley  and it is solution of the equation $F(x)=u $. We have 

\begin{eqnarray*}
1-\beta ^{-1}\left( \beta +\theta x\right) \exp \left( -\theta x\right) &=&u
\\
\beta ^{-1}\left( \beta +\theta x\right) \exp \left( -\theta x\right) &=&1-u
\\
\left( \beta +\theta x\right) \exp \left( -\theta x\right) &=&\beta \left(
1-u\right) \\
\left( \beta +\theta x\right) \exp \left( -\theta x\right) \exp \left(
-\beta \right) &=&\beta \left( 1-u\right) \exp \left( -\beta \right) \\
-\left( \beta +\theta x\right) \exp \left( -\beta -\theta x\right) &=&\beta
\left( u-1\right) \exp \left( -\beta \right) \\
-\beta -\theta x &=&W_{-1}\left( \beta \left( u-1\right) \exp \left( -\beta
\right) \right) \\
-\theta x &=&\beta +W_{-1}\left( \beta \left( u-1\right) \exp \left( -\beta
\right) \right) \\
x &=&-\frac{\beta }{\theta }-\frac{1}{\theta }W_{-1}\left( \beta \left(
u-1\right) \exp \left( -\beta \right) \right).
\end{eqnarray*}

\Bin Thus, the quantile function of Pseudo-Lindley distribution is 

\begin{equation}
x(u)=-\frac{\beta }{\theta }-\frac{1}{\theta }W_{-1}\left( \beta \left(
u-1\right) \exp \left( -\beta \right) \right).
\end {equation}

\Ni For $\alpha \in ] 0, +\infty [\setminus \{1\}$, the quantile function of the \textit{PL-APT}
 distribution is the solution of the equation $G_{\alpha}(x)=u$. we have
 
\begin{eqnarray*}
\frac{1-\alpha ^{1-\beta ^{-1}\left( \beta +\theta x\right) \exp \left(
-\theta x\right) }}{1-\alpha } &=& u \\
1-\alpha ^{1-\beta ^{-1}\left( \beta +\theta x\right) \exp \left( -\theta
x\right) } &=& u \left( 1-\alpha \right) \\
\alpha ^{1-\beta ^{-1}\left( \beta +\theta x\right) \exp \left( -\theta
x\right) } &=&1-u \left( 1-\alpha \right) \\
\left[ 1-\beta ^{-1}\left( \beta +\theta x\right) \exp \left( -\theta
x\right) \right] \log \alpha &=&\log \left[ 1-u \left(
1-\alpha \right) \right] \\
1-\beta ^{-1}\left( \beta +\theta x\right) \exp \left( -\theta x\right)
&=&\left( \log \alpha \right) ^{-1}\log \left[ 1-u \left(
1-\alpha \right) \right] \\
\beta ^{-1}\left( \beta +\theta x\right) \exp \left( -\theta x\right)
&=&1-\left( \log \alpha \right) ^{-1}\log \left[ 1-u \left(
1-\alpha \right) \right] \\
\left( \beta +\theta x\right) \exp \left( -\theta x\right) &=&\beta -\beta
\left( \log \alpha \right) ^{-1}\log \left[ 1-u \left(
1-\alpha \right) \right] \\
\left( \beta +\theta x\right) \exp \left( -\beta -\theta x\right) &=&\beta
\exp \left( -\beta \right) -\beta \exp \left( -\beta \right) \left( \log
\alpha \right) ^{-1}\log \left[ 1-u \left( 1-\alpha \right) %
\right].
\end{eqnarray*}

\Bin We apply the negative Lambert function on both sides to obtain

\begin{eqnarray*}
-\beta -\theta x &=&W_{-1}\left\{ \beta \exp \left( -\beta \right) -\beta
\exp \left( -\beta \right) \left( \log \alpha \right) ^{-1}\log \left[
1-u \left( 1-\alpha \right) \right] \right\} \\
-\theta x &=&\beta +W_{-1}\left\{ \beta \exp \left( -\beta \right) -\beta
\exp \left( -\beta \right) \left( \log \alpha \right) ^{-1}\log \left[
1-u \left( 1-\alpha \right) \right] \right\} \\
x &=&-\frac{\beta }{\theta }-\frac{1}{\theta }W_{-1}\left\{ \beta \exp
\left( -\beta \right) -\beta \exp \left( -\beta \right) \left( \log \alpha
\right) ^{-1}\log \left[ 1-u \left( 1-\alpha \right) \right]
\right\}. \\
&&
\end{eqnarray*}

\Bin Finally , we have

\begin{equation*}
x\left( u\right) =-\frac{\beta }{\theta }-\frac{1}{\theta }W_{-1}\left\{
\beta \exp \left( -\beta \right) -\beta \exp \left( -\beta \right) \left(
\log \alpha \right) ^{-1}\log \left[ 1-u \left( 1-\alpha
\right) \right] \right\}.
\end{equation*}


\begin{thebibliography}{9}

\bibitem[Aldahlan (2020)]{Aldahlan} Aldahlan M.A. (2020) \emph{ Alpha Power Transformed Log-Logistic Distribution with Application to Breaking Stress Data}. Hindawi Advances in Mathematical Physics Volume 2020, Article ID 2193787, 9 pages
 https://doi.org/10.1155/2020/2193787
  
\bibitem[Deme \textit{et al.} (2012)]{demedioplo} Deme, E., LO, G. S., and Diop, A. , (2012). \emph{ On the generalized Hill process for small parameters and applications}. Journal of Statistical Theory and Applications, 11(4), 397-418. http://dx.doi.org/10.2991/jsta.2013.12.1.3. (MR3191797)

\bibitem[Eghwerido (2021)]{eghwerido} Eghwerido J. T. ,(2021) . \emph{The alpha power Teissier distribution and its applications}. Afrika Statistika, Vol. 16 (2), 2021, pages 2731 - 2745, DOI: http://dx.doi.org/10.16929/as/2021.2731.181.


\bibitem[Ghitany et al. (2008)]{Ghitanyetal} Ghitany, M.E., Al-Mutairi, D.K., Aboukhamseen, S.M. (2008) \emph{ Estimation of the Reliability of a Stress-Strength System from Power Lindley Distributions}. Commun. Stat. Simul. Comput. 2015, 44, 118–136.


\bibitem [Gomez et al. (2014)]{Gomezetal}  Gomez-Déniz, E., Sordo, M.A., Calderín-Ojeda, E. (2014)  \emph{ The log–Lindley distribution as an alternative to the beta regression model with applications in insurance}. Insur. Math. Econ. 2014, 54, 49–57. 

\bibitem[Hafez et al. (2020)]{Hafezetal} Hafez E.H., Riad F.H., Mubarak Sh. A.M. and  Mohamed  M. S., (2020). \emph{Study on Lindley Distribution Accelerated Life Tests:Application and Numerical Simulation}. Symmetry 2020, 12, 2080; doi:10.3390/sym12122080.

\bibitem[Hill (1975)]{hill1975} Hill B.,(1975).\emph{A simple general approach to the inference about the tail index of a distribution}.
 Ann. Statist., Vol 3 (5), pp. 1163–1174. (MR0378204)

\bibitem[Ijaz et al. (2021)]{Ijazetal} Ijaz M., Mashwani W.K., Göktaş A. and Unvan Y.A. (2021). \emph{ A novel alpha power transformed exponential distribution with real-life applications}. Journal of Applied Statistics To link to this article: https://doi.org/10.1080/02664763.2020.1870673

 \bibitem[Irshad et al. (2021)]{Irshadetal} Irshad M. R., Chesneau C., D’cruz V. and  Maya R. (2021). \emph{Discrete Pseudo Lindley Distribution: Properties, Estimation and Application on INAR (1) Process}. Math. Comput. Appl. 2021, 26, 76. https://doi.org/10.3390/mca26040076


\bibitem[Krishna and Kumar (2011)]{KrishnaKumar2011} Krishna, H. and Kumar, K. (2011). \emph{Reliability estimation in Lindley 
distribution with Progressive type II right censored sample}. Journal Mathematics and Computers in Simulation archive, 82, 2, 281-294.

\bibitem[Lo \textit{et al.} (2021-2016)]{ips-wcia-ang} Lo G.S., Ngom M., Kpanzou T.A. and Niang A.B. (2021). \textit{Weak Convergence (IA). Sequences of random vectors}. Version pre-print (2016), Doi : 10.16929/sbs/2016.0001. Arxiv : 1610.05415. ISBN : 978-2-9559183-1-9. Version (paper and Euclid), SPAS Books Series. Saint-Louis, Senegal - Calgary, Canada. Published November 2021. Doi : http://dx.doi.org/10.16929/sts/2021.0003. ISBN (amazon): 9798761805149

\bibitem[Lo \textit{et al.} (2020)]{gslo2020} Lo G.S., Ngom M., Diallo M. (2020) . \emph{Extremes, extremal index estimation, records, moment problem for the Pseudo-Lindley distribution and applications}. European Journal of Pure and Applied Mathematics (EJPAM) Vol. 13, No. 4, 2020, 739-757 ISSN 1307-5543, https://www.ejpam.com. DOI: https://doi.org/10.29020/nybg.ejpam.v13i4.3834

\bibitem[Lo \textit{et al.} (2019)]{gslo2019} Lo G.S., Kpanzou T.A., Haidara C.M. (2019) \emph{Statistical tests for the Pseudo-Lindley distribution and applications}. Afrika Statistika Vol. 14 No. 4, 2019, pages 2127 - 2139. DOI: http://dx.doi.org/10.16929/as/2019.2127.151

\bibitem[Lo (2018)]{ips-mfpt-ang} Lo G.S. (2018). \textit{Mathematical Foundation of Probability Theory}. \text {SPAS Books Series}. Saint-Louis, Senegal - Calgary, Canada. Doi: http://dx.doi.org/10.16929/sbs/2016.0008. Arxiv: arxiv.org/pdf/1808.01713

\bibitem[Lindley (1958)]{lindley1958} Lindley D.V. (1958) \emph{Fiducial distributions and Bayes' theorem}.Journal of the Royal Statistical Society, Series B 20 (1958) 102–107.  (MR0095550)

\bibitem[Lindley (1965)]{lindley1965} Lindley D.V. (1965) \emph{Introduction to Probability and Statistics from a Bayesian Viewpoint}.Part II: Inference, Cambridge University Press, New York, 1965. (MR0168084)


\bibitem[Mazucheli and Achcar (2011)] {MazucheliAchcar2011} Mazucheli, J. and Achcar, J.A. (2011).  \emph{The Lindley distribution ap plied to competing risks lifetime data}. Computer Methods Programs in Biomedicine, 104, 2, 188-192.


\bibitem[Mahdavi and Kundu (2017)]{mahdavi2017} Mahdavi  A. and Kundu D. (2017). \emph{A new method for generating distribution with an application to exponential distribution}. Communications in statistics - Theory and methods  46(13),6543-6557, Published by New York Business Global. (MR3631530)

\bibitem[Ngom and Lo (2016)]{ngomlo2016} Ngom M. and Lo G.S., (2016) . \emph{A double-indexed functional hill process and applications}.Journal of Mathematical Research
(e-ISSN 1916-9809). Vol. 8 (4), pp. 144 165,2016, Doi : 105539/jmr/v8n4p144,

\bibitem[Unyime and Ette (2021)]{unyimeette}  Unyime P. U. and Ette H. E. (2021). Alpha power transformed quasi lindley distribution. \emph{International Journal of Advanced Statistics and Probability}, Vol. 9 (1), pp. 6-17.

\bibitem[Shanker et al. (2015)]{Shankeretal2015} Shanker R., Hagos F. and Sujatha S. (2015). \emph{On modeling of Lifetimes data using exponential and Lindley distributions}. Biometrics and Biostatistics International Journal, 2, 5, 1-9.

\bibitem[Shanker and Mishra (2013)]{ShankerMishra2013}[Shanker R. and Mishra A. (2013). \emph{A quasi Lindley distribution} African journal of mathematics and computer science research. Vol.6(4), pp. pp. 64-71 , April 2013 https://doi.org/10.5897/AJMCSR12.067.

\bibitem[Zeghdoudi and Nedjar (2016)]{zeghdoudi2016} Zeghdoudi H. and Nedjar S.,(2016) \emph{A pseudo lindley-distribution and its application}. Afrika Statistika, 11 (1):923–932. DOI: http://dx.doi.org/10.16929/as/2016.923.83.

\bibitem[ZeinEldin et al.(2021)]{ZeinEldinetal} ZeinEldin R.A., Haq M. A. U., Hashmi S. and Elsehety M. (2021) \emph{Alpha Power Transformed Inverse Lomax Distribution with Different Methods of Estimation and Applications}. Hindawi Complexity Volume 2020, Article ID 1860813, 15 pages;  https://doi.org/10.1155/2020/1860813.
\end{thebibliography}
\end{document}